\documentclass[a4paper]{article}
\usepackage{mathrsfs, amsmath, amsxtra, amssymb, latexsym, amscd, amsthm}

\newtheorem{thm}{Theorem}[section]
\newtheorem{cor}{Corollary}[section]
\newtheorem{lem}{Lemma}[section]
\newtheorem{prop}{Proposition}[section]
\theoremstyle{definition}
\newtheorem{defn}{Definition}[section]
\newtheorem{exam}{Example}[section]

\theoremstyle{remark}
\newtheorem{rem}{Remark}[section]
\numberwithin{equation}{section}
\def\ind{{\rm 1\hspace{-0.90ex}1}}

\begin{document}

\title{Poisson and normal approximations for the measurable functions of independent random variables}
\author{Nguyen Tien Dung\footnote{Email: dung\_nguyentien10@yahoo.com}}

\date{\today}          
\maketitle
\begin{abstract}
In this paper we use a Malliavin-Stein type method to investigate Poisson and normal approximations for the measurable functions of infinitely many independent random variables. We combine Stein's method with the difference operators in theory of concentration inequalities to obtain explicit bounds on Wasserstein, Kolmogorov and total variation distances. When restricted to the functions of a finite number of independent random variables, our method provides new bounds in the normal approximation. Meanwhile, our bounds in Poisson approximation are first to obtain explicitly.
\end{abstract}
\noindent\emph{Keywords:} Stein's method, Difference operators, Wasserstein distance, Kolmogorov distance, Total variation distance.\\
{\em 2010 Mathematics Subject Classification:} 60F05, 60B10.

\section{Introduction}
Since the appearance of two seminal papers, Nualart \& Ortiz-Latorre \cite{NualartLatorre2008} and Nourdin \& Peccati \cite{Nourdin2009}, a new research line has been established. In this context, one combines Stein’s method with the Malliavin calculus to improve and refine many results in the normal approximation for functionals of Gaussian processes. Nowadays, this research line is the so-called Malliavin-Stein method and in the last decade, many important achievements have been obtained by various authors. For an overview, we refer the reader to the website $$\text{https://sites.google.com/site/malliavinstein/home.}$$
In particular, Malliavin-Stein method has been successfully used to study probability approximations for Rademacher functionals of the form
$$F(\varepsilon)=F(\varepsilon_1,\varepsilon_2,...),$$
where $\varepsilon:=(\varepsilon_1,\varepsilon_2,...)$ is an infinite sequence of independent Rademacher random variables.  More specifically, the normal approximation for $F(\varepsilon)$ has been investigated in \cite{Nourdin2010,Privault2015} and in recent papers \cite{Dobler2017,Krokowski2016,Zheng2016}. Poisson approximation for $F(\varepsilon)$ has also been investigated in \cite{Privault2015} and in \cite{Krokowski2017}. The error bounds in Poisson and normal approximations obtained in these papers are determined in terms of discrete Malliavin derivative operator $D$ (see \cite{Privault2008} for the original reference). The power of Malliavin-Stein method lies in the facts that it can handles the infinite sequences and provides the explicit bounds.

Let $X=(X_1,X_2,...)$ be a sequence of independent random variables (not necessarily identically distributed). We consider the problem of probability approximations for functionals of the form
\begin{equation}\label{11lqo}
F:=F(X)=F(X_1,X_2,...).
\end{equation}
When the functional depends only on the first $n$ coordinates, the normal approximation for $F(X_1,X_2,...,X_n)$ has been studied by Chatterjee \cite{Chatterjee2008} and L-Rey \& Peccati \cite{Rey2017}. However, it is still an open problem for the case of infinite sequences.


The aim of this paper is to introduce a new difference operator standing for the Malliavin derivative operator $D$ so that we can develop a Malliavin-Stein type method to study probability approximations for (\ref{11lqo}). The idea behind our work comes from two well-known results in the literature: the first one is due to  Chatterjee's work \cite{Chatterjee2008} and the second one is Efron-Stein inequality stated in Theorem 3.1 of \cite{Boucheron2013}.

For the measurable functions $f(X)=f(X_1,X_2,...,X_n)$ of $n$ arbitrary independent random variables ($n<\infty$), Chatterjee introduced in \cite{Chatterjee2008} a new method of normal approximation to obtain the explicit bounds on Wasserstein distance. His method can be summarized as follows: Let $X'=(X'_1,X'_2,...,X'_n)$ be an independent copy of $X=(X_1,X_2,...,X_n).$ For each $A\subseteq [n]=\{1,2,...,n\},$ define the random vector $X^A$ as
$$X^A_i=\left\{
          \begin{array}{ll}
            X'_i, & \hbox{if}\,\,\, i\in A,\\
            X_i, & \hbox{if}\,\,\, i\notin A.
          \end{array}
        \right.
$$
For each $j\in [n],$ we write $X^j$ instead of $X^{\{j\}}$ and define the difference operator $\Delta_j$ by
\begin{equation}\label{p95j4}
\Delta_jf(X):=f(X)-f(X^j).
\end{equation}
Next, for each $A\subseteq [n]=\{1,2,...,n\},$ let
$$T_A:=\sum\limits_{j\notin A}\Delta_jf(X)\Delta_jf(X^A)\,\,\,\text{and}\,\,\,T:=\frac{1}{2}\sum\limits_{A\subsetneq [n]}\frac{T_A}{\binom{n}{|A|}(n-|A|)}.$$
The abstract result stated in Theorem 2.2 of \cite{Chatterjee2008} reads: Suppose that the random variable $f(X)$ has mean zero and variance $\sigma^2\in(0,\infty).$ Then, the Wasserstein distance between the law of $\sigma^{-1}f(X)$ and standard normal law $N$ satisfies
\begin{equation}\label{kodw3a}
d_W(\sigma^{-1}f(X),N)\leq \frac{\sqrt{Var(E[T|X])}}{\sigma^2}+\frac{1}{2\sigma^3}\sum\limits_{j=1}^nE|\Delta_jf(X)|^3.
\end{equation}


Let us now recall the Efron-Stein inequality stated in Theorem 3.1 of \cite{Boucheron2013}: Suppose that $f(X)$ is a square-integrable random variable, then
$$Var(f(X))\leq E\left[\sum\limits_{j=1}^n (f(X)-E_j[f(X)])^2\right]=\frac{1}{2}E\left[\sum\limits_{j=1}^n (f(X)-f(X^j))^2\right],$$
where $E_j$ denotes the expectation with respect to $X_j.$

Because of the appearance of the factor $\binom{n}{|A|}(n-|A|)$ in the definition of $T,$ Chatterjee's method can not be extended to the functionals of infinitely many independent random variables of the form (\ref{11lqo}). However, we observe that the difference $f(X)-f(X^j)$ was used by Chatterjee to define the operator $\Delta_j.$ Hence, we wonder that if we can use $f(X)-E_j[f(X)]$ to define a new operator, namely $\mathfrak{D}_j,$ and combine this operator with Stein's method to investigate the normal approximation. Fortunately,  the answer is affirmative.

Our Theorem \ref{ko67d3} below provides the following bound on Wasserstein distance between the law of $\sigma^{-1}f(X)$ and standard normal law $N:$
\begin{equation}\label{kodw3b}
d_W(\sigma^{-1}f(X),N)\leq \frac{\sqrt{Var(Z)}}{\sigma^2}+\frac{2}{\sigma^3}\sum\limits_{j=1}^nE|\mathfrak{D}_jf(X)|^3,
\end{equation}
where $Z=\sum\limits_{j=1}^n  \mathfrak{D}_jf(X) E[\mathfrak{D}_jf(X)|\mathcal{F}_j]$ and $\mathcal{F}_j=\sigma(X_k,k\leq j).$ Since we use the same techniques of Stein's method, our bound (\ref{kodw3b}) is similar to (\ref{kodw3a}) with $Z$ and $\mathfrak{D}_j$ play the role of $E[T|X]$ and $\Delta_j,$ respectively. At the moment, we do not know which of the bounds (\ref{kodw3a}) and (\ref{kodw3b}) is easier to use in practice. But, at least, our bound (\ref{kodw3b}) provides a new way to prove central limit theorems. Another interesting feature of the operator $\mathfrak{D}_j$ is that it can handle the functionals of infinitely many independent random variables (\ref{11lqo}). Those two observations encourage us to write the present paper.


Developing further our work, we find out that the operator $\mathfrak{D}_j$ can also be used to obtain the explicit bounds in Poisson approximation. In the context of the functions of independent random variables, to the best of our knowledge, such explicit bounds are first to obtain.

The rest of this article is organized as follows. In Section 2, we recall the definition of two certain difference operators in theory of concentration inequalities and construct a new covariance formula. We also introduce in this section the concept of generalized Lyapunov ratios which will be used to represent our bounds.

In Section 3, we obtain the explicit bounds on Wasserstein and Kolmogorov distances in the normal approximation for the functionals (\ref{11lqo}). Our abstract findings are formulated in Theorems \ref{ko67d3} and \ref{mld2sk}. In Theorem \ref{luongmoi} we provide a slight generalization of classical results to infinite sums. The bounds on Kolmogorov distance, which are more convenient to use in practice, are provided in Corollaries \ref{oold1}, \ref{oold1b} and \ref{o7old1q}. In this section, we also show that our abstract bounds are pretty easy to apply to the sums of locally dependent random variables.

Section 4 is devoted to Poisson approximation in Wasserstein and total variation distances. The explicit bounds on these distances are stated and proved in Theorems \ref{prove01} and \ref{prove02}. 


\section{Covariance formula based on difference operators}
Let $\mathcal{X}$ be a measurable space and $X=(X_1,X_2,...)$ be a sequence of independent random variables, defined on some probability space $(\Omega,\mathfrak{F},P)$ and taking values in $\mathcal{X}.$ For each $\mathbb{R}$-valued measurable function $F,$ we consider the random variable $F:=F(X).$ Let $X'=(X'_1,X'_2,...)$ be an independent copy of $X.$ We write $T_iF=F(X_1,...,X_{i-1},X'_i,X_{i+1},...),i\geq 1$ and denote by $E_i,E'_i$ the expectations with respect to $X_i$ and $X'_i,$ respectively.

We first recall the definition of two certain difference operators in theory of concentration inequalities. Here we follow the notations used in \cite{Bobkov2017}.
\begin{defn}\label{kod2ol} Given a random variable $F=F(X)\in L^1(P),$ we define the difference operators $\mathfrak{D}_i$ by
$$\mathfrak{D}_iF=F-E_i[F],\,\,i\geq 1.$$
When $F\in L^2(P),$ we define the difference operators $\mathfrak{d}_i$ by
$$\mathfrak{d}_iF=\big(\frac{1}{2}E'_i|F-T_iF|^2\big)^{\frac{1}{2}},\,\,i\geq 1.$$
\end{defn}
Let us now prepare some useful properties of the operators $\mathfrak{D}_i$ and $\mathfrak{d}_i.$ We introduce the $\sigma$-fields
\begin{align*}
&\mathcal{F}_0:=\{\emptyset,\Omega\}\,\,\,\text{and}\,\,\,\mathcal{F}_i:=\sigma(X_k,k\leq i),\,\,i\geq 1.
\end{align*}
\begin{prop}\label{mcsk4} For each $i\geq 1,$ under suitable integrability assumptions, we have

(i) $E[\mathfrak{D}_iF]=0,$ 


(ii) $\mathfrak{D}_iE[F|\mathcal{F}_i]=E[F|\mathcal{F}_i]-E[F|\mathcal{F}_{i-1}]=E[\mathfrak{D}_iF|\mathcal{F}_i],$

(iii) $E\left[(\mathfrak{D}_iF)G\right]=E\left[(\mathfrak{D}_iG)F\right]=E\left[(\mathfrak{D}_iF)(\mathfrak{D}_iG)\right],$

(iv) $(\mathfrak{d}_iF)^2=\frac{1}{2}[(\mathfrak{D}_iF)^2+E_i(\mathfrak{D}_iF)^2],$

(v) $\mathfrak{D}_i(FG)=F\mathfrak{D}_iG+G\mathfrak{D}_iF-\mathfrak{D}_iF\mathfrak{D}_iG-E_i[\mathfrak{D}_iF\mathfrak{D}_iG],$

(vi) $E|\mathfrak{D}_iF|^p\leq 2^p E|F|^p,\,\,\forall\,p\geq 1.$
\end{prop}
\begin{proof} The point $(i)$ follows directly from the definition of $\mathfrak{D}_i.$

\noindent$(ii)$ By the independence, we have $E_i[F]=E[F|\sigma(X_k,k\neq i)].$ Hence, we obtain
$$\mathfrak{D}_iE[F|\mathcal{F}_i]=E[F|\mathcal{F}_i]-E[E[F|\mathcal{F}_i]|\sigma(X_k,k\neq i)]=E[F|\mathcal{F}_i]-E[F|\mathcal{F}_{i-1}]$$
and
$$E[\mathfrak{D}_iF|\mathcal{F}_i]=E[F|\mathcal{F}_i]-E[E[F|\sigma(X_k,k\neq i)]|\mathcal{F}_i]=E[F|\mathcal{F}_i]-E[F|\mathcal{F}_{i-1}].$$
\noindent$(iii)$ This point follows from the relation
$$E[E_i[F]E_i[G]]=E[FE_i[G]]=E[E_i[F]G].$$
\noindent$(iv)$ Because $E_i[F]=E'_i[T_iF],$ we have
\begin{align*}
E'_i[(E_iF-T_iF)^2]&=(E_iF)^2-2(E_iF)^2+E_i[F^2]\\
&=E_i[F^2]-(E_iF)^2\\
&=E_i[(F-E_iF)^2]\\
&=E_i(\mathfrak{D}_iF)^2.
\end{align*}
This, together with the decomposition
$(F-T_iF)^2=(F-E_iF)^2+2(F-E_iF)(E_iF-T_iF)+(E_iF-T_iF)^2,$ gives us
$$2(\mathfrak{d}_iF)^2=(\mathfrak{D}_iF)^2+E'_i[(E_iF-T_iF)^2]=(\mathfrak{D}_iF)^2+E_i(\mathfrak{D}_iF)^2.$$
\noindent$(v)$ We have
\begin{align*}
FG&-T_iFT_iG=F(G-T_iG)+G(F-T_iF)-(F-T_iF)(G-T_iG)\\
&=F(G-T_iG)+G(F-T_iF)-\frac{(F+G-T_iF-T_iG)^2-(F-G-T_iF+T_iG)^2}{4}.
\end{align*}
Hence, we obtain
\begin{align*}
\mathfrak{D}_i(FG)=F\mathfrak{D}_iG+G\mathfrak{D}_iF-\frac{(\mathfrak{d}_i(F+G))^2-(\mathfrak{d}_i(F-G))^2}{2}.
\end{align*}
So we can finish the proof by using the point $(v).$

\noindent$(vi)$ By using the fundamental inequality $(a+b)^p\leq 2^{p-1}(a^p+b^p)$ we obtain
$$E|\mathfrak{D}_iF|^p\leq 2^{p-1} (E|F|^p+E|E_i[F]|^p)\leq 2^{p}E|F|^p,\,\,p\geq 1.$$
The proof of Proposition is complete.
\end{proof}
\begin{thm}\label{lods3} Let $F=F(X)$ and $G=G(X)$ be two random variables in $L^2(P),$ we have
\begin{equation}\label{damnit}
Cov(F,G)=E\left[\sum\limits_{i=1}^\infty \mathfrak{D}_i F E[\mathfrak{D}_iG|\mathcal{F}_i]\right].
\end{equation}
\end{thm}
\begin{proof}We have $(E[G|\mathcal{F}_n])_{n\geq 1}$ is a martingale satisfying the decomposition
\begin{align}
E[G|\mathcal{F}_n]-E[G|\mathcal{F}_0]&=\sum\limits_{i=1}^n(E[G|\mathcal{F}_i]-E[G|\mathcal{F}_{i-1}])\,\,\forall\,\,n\geq 1.\label{infty01}
\end{align}
Consider the random variable $U:=E[G|\mathcal{F}_n].$ Then, $U$ is a function of $n$ independent random variables $(X_1,...,X_n).$ It is known from the page 54, line 1 in \cite{Boucheron2013} that
$$Var(U)=\sum\limits_{i=1}^nE(E[U|\mathcal{F}_i]-E[U|\mathcal{F}_{i-1}])^2.$$
This, together with the fact that $E[U|\mathcal{F}_i]=E[G|\mathcal{F}_i],\,\,i\leq n,$ yields
$$\sum\limits_{i=1}^nE(E[G|\mathcal{F}_i]-E[G|\mathcal{F}_{i-1}])^2=Var(U)\leq Var(G)<\infty\,\,\forall\,\,n\geq 1.$$
So, the series $\sum\limits_{i=1}^\infty E(E[G|\mathcal{F}_i]-E[G|\mathcal{F}_{i-1}])^2$ is convergent. By martingale convergence theorem, the relation (\ref{infty01}) gives us
\begin{align*}
G-E[G]&=\sum\limits_{i=1}^\infty(E[G|\mathcal{F}_i]-E[G|\mathcal{F}_{i-1}])\\
&=\sum\limits_{i=1}^\infty \mathfrak{D}_iE[G|\mathcal{F}_i]\,\,\text{by Proposition \ref{mcsk4}, $(ii)$}.
\end{align*}
Hence, we can get
\begin{align*}
Cov(F,G)&=E[F(G-E[G])]\\
&=\sum\limits_{i=1}^\infty E\left[F\mathfrak{D}_iE[G|\mathcal{F}_i]\right]\\
&=\sum\limits_{i=1}^\infty E\left[\mathfrak{D}_iF\mathfrak{D}_iE[G|\mathcal{F}_i]\right]\,\,\text{by Proposition \ref{mcsk4}, $(iii)$}\\
&=\sum\limits_{i=1}^\infty E\left[\mathfrak{D}_iFE[\mathfrak{D}_iG|\mathcal{F}_i]\right]\,\,\text{by Proposition \ref{mcsk4}, $(ii)$}.
\end{align*}
The proof is complete.
\end{proof}
\begin{cor}\label{kkjj8}For any $F=F(X)\in L^2(P),$ we have
$$Var(F)=E\left[\sum\limits_{i=1}^\infty \mathfrak{D}_i F E[\mathfrak{D}_iF|\mathcal{F}_i]\right]=E\left[\sum\limits_{i=1}^\infty (E[\mathfrak{D}_iF|\mathcal{F}_i])^2\right].$$
\end{cor}
Assuming that the random variables $X_i,i\geq 1$ have the means $\mu_i=E[X_i]$ and finite variances $\sigma_i^2=E|X_i-\mu_i|^2,$ we consider the normalized partial sum
\begin{equation}\label{bbdsm2}
S_n=\Sigma^{-1/2}_n\sum\limits_{i=1}^n(X_i-\mu_i),
\end{equation}
where $\Sigma_n=\sum\limits_{i=1}^n\sigma_i^2$ and the Lyapunov ratio of order $r>0,$
$$L_r=\Sigma_n^{-r/2}\sum\limits_{i=1}^nE|X_i-\mu_i|^r.$$
The classical results (see, e.g. \cite{Chen2011}) tell us that ones can use the Lyapunov ratios to represent the bounds in the normal approximation for $S_n.$ For example, we have the following
\begin{equation}\label{yyh3}
d_W(S_n,N)\leq cL_3,\,\,\,\,d_K(S_n,N)\leq cL_3,
\end{equation}
where $d_W,d_K$ denote the Wasserstein and Kolmogorov distances and $c$ is an absolute constant. Naturally, one would like to obtain such familiar representations for the general functionals (\ref{11lqo}). For this purpose, let us introduce the following concept.

\begin{defn}Given a random variable $F=F(X),$ we define its generalized Lyapunov ratio of order $r>0$ by
$$\mathcal{L}_r(F):=\sum\limits_{i=1}^\infty E|\mathfrak{D}_iF|^r.$$
\end{defn}
When $F=S_n,$ $\mathcal{L}_r$ reduces to standard Lyapunov ratio. Indeed, we have $\mathfrak{D}_iS_n=\frac{X_i-\mu_i}{\sqrt{\Sigma_n}},i=1,2,...,n$ and hence,
$$\mathcal{L}_r(S_n)=\Sigma_n^{-r/2}\sum\limits_{i=1}^nE|X_i-\mu_i|^r=L_r.$$
In particular, when $S_n$ is a sum of independent and identically distributed (i.i.d.) random variables, we have
$$L_r=\frac{E|X_1-\mu_1|^r}{(E|X_1-\mu_1|^2)^{r/2}}\frac{1}{n^{r/2-1}}=O(\frac{1}{n^{r/2-1}}),\,\,n\to\infty.$$
Because $E(E[\mathfrak{D}_iF|\mathcal{F}_i])^2\leq E|\mathfrak{D}_iF|^2$ we obtain from Corollary \ref{kkjj8} the following.
\begin{cor} (Efron-Stein inequality) For any $F=F(X)\in L^2(P),$ we have
\begin{equation}\label{sopq1}
Var(F)\leq \mathcal{L}_2(F)=\sum\limits_{i=1}^\infty E|\mathfrak{D}_iF|^2.
\end{equation}
\end{cor}
It should be noted that  the Efron-Stein inequality is extremely useful for bounding the variances appearing in our bounds.


\section{Stein's method for normal approximation}
Before stating our main results, let us give here some remarks. We learn from the referee's reports on the previous version of this paper that the covariance formula (\ref{damnit}) was already obtained by Decreusefond \& Halconruy, see Theorem 3.6 in \cite{Decreusefond2017}. Although our proof is not the same as that of theirs, we would like to claim that our Theorem \ref{lods3} is not new anymore. In addition, we refer the reader to Section 5.2 in \cite{Decreusefond2017} for the normal approximation results obtained there.



\subsection{Wasserstein distance}
We first recall some fundamental results about Stein's method of the normal approximation. The Wasserstein distance between the law of $F$ and standard normal law $N$ is defined by
$$d_W(F,N):=\sup\limits_{|h(x)-h(y)|\leq |x-y|}|E[h(F)]-E[h(N)]|=\sup\limits_{h\in \mathcal{C}^1,\|h'\|_\infty\leq 1}|E[h(F)]-E[h(N)]|,$$
where $\|.\|_\infty$ denotes the supremum norm.

Given an absolutely continuous $h$ with bounded $h',$ we consider the Stein equation
\begin{equation}\label{fklls3}
f'(z)-zf(z)=h(z)-E[h(N)],\,\,z\in \mathbb{R}.
\end{equation}
It is known from Lemma 2.4 in \cite{Chen2011} that the equation (\ref{fklls3}) admits an unique solution, denoted by $f_h(z),$ and this solution satisfies
$$\|f_h\|_\infty\leq 2\|h'\|_\infty,\,\,\,\|f'_h\|_\infty\leq \sqrt{\frac{2}{\pi}}\|h'\|_\infty,\,\,\,\|f''_h\|_\infty\leq 2\|h'\|_\infty.$$
We now observe that
$$E[h(F)]-E[h(N)]=E[f'_h(F)]-E[Ff_h(F)].$$
Hence, the Wasserstein distance can be estimated as follows
\begin{equation}\label{nnsd7f}
d_W(F,N)\leq \sup\limits_{f\in \mathcal{F}_W}|E[f'(F)]-E[Ff(F)]|,
\end{equation}
where $\mathcal{F}_W$ is the class of differentiable functions $f$ satisfying $\|f'\|_\infty\leq \sqrt{\frac{2}{\pi}}$ and $\|f''\|_\infty\leq 2.$

In order to be able to combine the covariance formula obtained in Theorem \ref{lods3} with Stein's method, let us provide a chain rule for the difference operators $\mathfrak{D}_i.$
\begin{lem}\label{lods33} Let $f:\mathbb{R}\to \mathbb{R}$ be a differentiable function with bounded derivative such that $f'$ is Lipschitz continuous. For any $F=F(X)\in L^2(P),$ we have
\begin{align*}
\mathfrak{D}_i f(F)&=f'(F) \mathfrak{D}_iF+R_{i,f},\,\,i\geq 1,
\end{align*}
where the remainder term $R_{i,f}$ satisfies the bound
$$|R_{i,f}|\leq \|f''\|_\infty(\mathfrak{d}_iF)^2=\frac{\|f''\|_\infty}{2}[(\mathfrak{D}_iF)^2+E_i(\mathfrak{D}_iF)^2],\,\,i\geq 1.$$
\end{lem}
\begin{proof} By the Taylor expansion we have $f(x)-f(y)=f'(x)(x-y)+R_f,$ where the remainder term $R_f$ is bounded by $\|f''\|_\infty(x-y)^2/2$ for all $x,y\in \mathbb{R}.$

Hence, for each $i\geq 1,$ we have
\begin{align*}
\mathfrak{D}_i f(F)&=E'_i[f(F)-T_if(F)]\\
&=E'_i[f(F)-f(T_iF)]\\
&=E'_i[f'(F)(F-T_iF)]+R_{i,f}\\
&=f'(F) \mathfrak{D}_iF+R_{i,f},
\end{align*}
where
$$|R_{i,f}|\leq \frac{\|f''\|_\infty}{2}E'_i(F-T_iF)^2=\|f''\|_\infty(\mathfrak{d}_iF)^2=\frac{\|f''\|_\infty}{2}[(\mathfrak{D}_iF)^2+E_i(\mathfrak{D}_iF)^2].$$
The proof is complete.
\end{proof}
The next statement is the first main result of the present paper.
\begin{thm}\label{ko67d3}Let $F=F(X)$ be in $L^2(P)$ with mean zero, we have
\begin{align}
d_W(F,N)&\leq \sqrt{\frac{2}{\pi}}E\big|1-Z\big|+2E\left[\sum\limits_{i=1}^{\infty} (\mathfrak{d}_iF)^2\,|E[\mathfrak{D}_iF|\mathcal{F}_i]|\right]\label{lode4}\\
&\leq \sqrt{\frac{2}{\pi}}|1-E[F^2]|+\sqrt{\frac{2}{\pi}}\sqrt{Var(Z)}+2\mathcal{L}_3(F),\label{de2q1}
\end{align}
where $Z:=\sum\limits_{i=1}^{\infty}  \mathfrak{D}_iF E[\mathfrak{D}_iF|\mathcal{F}_i].$
\end{thm}
\begin{proof}For any $f\in \mathcal{F}_W,$ we have $f(F)\in L^2(P)$ and $E[Ff(F)]=Cov(F,f(F))$ because $E[F]=0.$ Thanks to Theorem \ref{lods3} and Lemma \ref{lods33} we obtain
\begin{align*}
E[Ff(F)]&=E\left[\sum\limits_{i=1}^{\infty} \mathfrak{D}_i f(F) E[\mathfrak{D}_iF|\mathcal{F}_i]\right]\\
&=E\left[\sum\limits_{i=1}^{\infty} f'(F) \mathfrak{D}_iF E[\mathfrak{D}_iF|\mathcal{F}_i]\right]+E\left[\sum\limits_{i=1}^{\infty} R_{i,f} E[\mathfrak{D}_iF|\mathcal{F}_i]\right]\\
&=E[f'(F)Z]+E\left[\sum\limits_{i=1}^{\infty} R_{i,f} E[\mathfrak{D}_iF|\mathcal{F}_i]\right].
\end{align*}
As a consequence,
\begin{align*}
|E[&f'(F)]-E[Ff(F)]|\\
&\leq E|f'(F)-f'(F)Z|+E\left[\sum\limits_{i=1}^{\infty} |R_{i,f}| |E[\mathfrak{D}_iF|\mathcal{F}_i]|\right]\\
&\leq \|f'\|_\infty E|1-Z|+\|f''\|_\infty E\left[\sum\limits_{i=1}^{\infty} (\mathfrak{d}_iF)^2\,|E[\mathfrak{D}_iF|\mathcal{F}_i]|\right]\,\,\forall\,f\in \mathcal{F}_W
\end{align*}
and so (\ref{lode4}) follows.
By using the H\"older inequality we deduce
$$E[(\mathfrak{D}_iF)^2|E[\mathfrak{D}_iF|\mathcal{F}_i]|]\leq (E|\mathfrak{D}_iF|^3)^{\frac{2}{3}}(E|E[\mathfrak{D}_iF|\mathcal{F}_i]|^3)^{\frac{1}{3}}\leq E|\mathfrak{D}_iF|^3,$$
$$E[E_i(\mathfrak{D}_iF)^2|E[\mathfrak{D}_iF|\mathcal{F}_i]|]\leq (E[E_i|\mathfrak{D}_iF|^3])^{\frac{2}{3}}(E|E[\mathfrak{D}_iF|\mathcal{F}_i]|^3)^{\frac{1}{3}}\leq E|\mathfrak{D}_iF|^3.$$
Hence, it holds that
\begin{align}
2E[(\mathfrak{d}_iF)^2\,|E[\mathfrak{D}_iF|\mathcal{F}_i]|]\leq 2E|\mathfrak{D}_iF|^3,\,\,i\geq 1.\label{ode3wy}
\end{align}
On the other hand, by using Corollary \ref{kkjj8} and  the triangle inequality, we have
\begin{align}
E\big|1-Z\big|&\leq |1-E[F^2]|+E|Z-E[Z]|\nonumber\\
&\leq |1-E[F^2]|+\sqrt{Var(Z)}.\label{ml2s7}
\end{align}
Inserting the estimates (\ref{ode3wy}) and (\ref{ml2s7}) into (\ref{lode4}) gives us (\ref{de2q1}). The proof of Theorem is complete.
\end{proof}

\begin{exam}We consider the partial sum $S_n$ defined by (\ref{bbdsm2}). For the simplicity, we now assume that $\mu_i=0,i=1,2,...,n.$ Then, we have $\mathfrak{D}_iS_n=\frac{X_i}{\sqrt{\Sigma_n}},$ $E[\mathfrak{D}_iS_n|\mathcal{F}_i]=\frac{X_i}{\sqrt{\Sigma_n}}$ and
$$Z= \Sigma^{-1}_n\sum\limits_{i=1}^{n}X^2_i.$$
We also have $\mathfrak{D}_iZ= \frac{X^2_i-\sigma_i^2}{\Sigma_n},i\geq 1.$ Hence, if all $X_i$ have finite fourth moments, the estimate (\ref{de2q1}) and the Efron-Stein inequality (\ref{sopq1}) give us
$$
d_W(S_n,N)\leq \sqrt{\frac{2\sum\limits_{i=1}^{n}E|X^2_i-\sigma_i^2|^2}{\pi\Sigma_n^2}}+2L_3
=\sqrt{\frac{2\sum\limits_{i=1}^{n}(E|X_i|^4-\sigma_i^4)}{\pi\Sigma_n^2}}+2L_3\leq  \sqrt{\frac{2L_4}{\pi}}+2L_3.
$$
To compare with Chatterjee's method, we recall that the detailed computations from \cite{Chen2011} (pages 117-119) yield
$$d_W(S_n,N)\leq \sqrt{\frac{\sum\limits_{i=1}^{n}(E|X_i|^4+3\sigma_i^4)}{\pi\Sigma_n^2}}
+\frac{\sum\limits_{i=1}^{n}(E|X_i|^4+3\sigma_i^4)^{3/4}}{2^{1/4}\Sigma_n^{3/2}}.$$
The above bounds both give the optimal rate of convergence $\frac{1}{\sqrt{n}}$ when $S_n$ is a sum of i.i.d. random variables. However, they do not recover the classical bound (\ref{yyh3}). In the next theorem, we use another chain rule to generalize this classical bound to infinite sums.

\end{exam}
\begin{thm}\label{luongmoi}Let $X_1,X_2...$ be a sequence of independent random variables with means zero and finite variances $\sigma_i^2=E|X_i|^2.$ Suppose that $\Sigma_\infty:=\sum\limits_{i=1}^\infty\sigma_i^2\in (0,\infty).$
Consider  the normalized series
 $$S_\infty= \Sigma^{-1}_\infty\sum\limits_{i=1}^{\infty}X_i.$$
If all $X_i$ have finite third absolute moments, then
\begin{equation}\label{nnsd7f77}
d_W(S_\infty,N)\leq4 \mathcal{L}_3(S_{\infty})=4\Sigma_\infty^{-3/2}\sum\limits_{i=1}^\infty E|X_i|^3.
\end{equation}
\end{thm}
\begin{proof} For each $f\in \mathcal{F}_W,$ we write
$$\mathfrak{D}_if(S_\infty)=E'_i[f(S_{\infty})-f(T_iS_{\infty})]=f(S_{\infty})-f(E_iS_{\infty})+E'_i[f(E_iS_{\infty})-f(T_iS_{\infty})].$$
By using the Taylor expansion as in the proof of Lemma \ref{lods33} we get
$$f(S_{\infty})-f(E_iS_{\infty})=f'(E_iS_{\infty})(S_{\infty}-E_iS_{\infty})+R^{\ast}_{i,f}$$
and
$$E'_i[f(E_iS_{\infty})-f(T_iS_{\infty})]=f'(E_iS_{\infty})E'_i[E_iS_{\infty}-T_iS_{\infty}]+R^{\ast\ast}_{i,f}=R^{\ast\ast}_{i,f},$$
where the remainders satisfy
$$|R^\ast_{i,f}|\leq \frac{\|f''\|_\infty}{2}(S_{\infty}-E_iS_{\infty})^2= \frac{\|f''\|_\infty}{2}(\mathfrak{D}_iS_{\infty})^2,$$
$$|R^{\ast\ast}_{i,f}|\leq \frac{\|f''\|_\infty}{2}E'_i(E_iS_{\infty}-T_iS_{\infty})^2=\frac{\|f''\|_\infty}{2}E_i(\mathfrak{D}_iS_{\infty})^2.$$
Thus we can write
$$\mathfrak{D}_if(S_{\infty})=f'(E_iS)\mathfrak{D}_iS_{\infty}+R_{i,f},\,\,i\geq 1$$
with the remainder $|R_{i,f}|\leq |R^\ast_{i,f}|+|R^{\ast\ast}_{i,f}|\leq \|f''\|_\infty(\mathfrak{d}_iS_{\infty})^2,\,\,i\geq 1.$

We have $\mathfrak{D}_iS_{\infty}=\frac{X_i}{\sqrt{\Sigma_\infty}},$ $E[\mathfrak{D}_iS_{\infty}|\mathcal{F}_i]=\frac{X_i}{\sqrt{\Sigma_\infty}}$ and $(\mathfrak{d}_iS_{\infty})^2=\frac{X_i^2+\sigma_i^2}{2\Sigma_\infty},\,\,i\geq 1.$ Hence, it follows from Theorem \ref{lods3} that
\begin{equation}\label{nnsd7f7}
E[f'(S_{\infty})]-E[S_{\infty}f(S_{\infty})]=E[f'(S_{\infty})]-E\left[\sum\limits_{i=1}^{\infty} f'(E_iS_{\infty}) \frac{X^2_i}{\Sigma_\infty}\right]-E\left[\sum\limits_{i=1}^{\infty} R_{i,f} \frac{X_i}{\sqrt{\Sigma_\infty}}\right].
\end{equation}
Because $E_iS_{\infty}$ and $X_i$ are independent, we obtain
\begin{align*}
E[f'(S_{\infty})]-E\left[\sum\limits_{i=1}^{\infty} f'(E_iS_{\infty}) \frac{X^2_i}{\Sigma_\infty}\right]&=E[f'(S_{\infty})]-\sum\limits_{i=1}^{\infty} E[f'(E_iS_{\infty})] \frac{\sigma^2_i}{\Sigma_\infty}\\
&=\sum\limits_{i=1}^{\infty} E[f'(S_{\infty})-f'(E_iS_{\infty})] \frac{\sigma^2_i}{\Sigma_\infty}
\end{align*}
and so
\begin{align*}
\big|E[f'(S_{\infty})]-E\left[\sum\limits_{i=1}^{\infty} f'(E_iS_{\infty}) \frac{X^2_i}{\Sigma_\infty}\right]\big|&\leq\|f''\|_\infty\sum\limits_{i=1}^{\infty} E|E_iS_{\infty}-S_{\infty}|\frac{\sigma^2_i}{\Sigma_\infty}\\
&=\|f''\|_\infty\sum\limits_{i=1}^{\infty} \frac{E|X_i|^2E|X_i|}{\Sigma_\infty^{3/2}}\\
&\leq 2\mathcal{L}_3(S_{\infty})\,\,\,\text{by Lyapunov's inequality}.
\end{align*}
We also have
\begin{align*}
\big|E\left[\sum\limits_{i=1}^{\infty} R_{i,f} \frac{X_i}{\sqrt{\Sigma_\infty}}\right]\big|&\leq \|f''\|_\infty\sum\limits_{i=1}^{\infty} \frac{ E|(\mathfrak{d}_iS_{\infty})^2X_i|}{\sqrt{\Sigma_\infty}}\\
&\leq \frac{\|f''\|_\infty}{2}\sum\limits_{i=1}^{\infty} \frac{ E|(X_i^2+\sigma_i^2)X_i|}{\Sigma_\infty^{3/2}}\\
&\leq 2\mathcal{L}_3(S_{\infty})\,\,\,\text{by Lyapunov's inequality}.
\end{align*}
Recalling (\ref{nnsd7f}) and (\ref{nnsd7f7}) we obtain (\ref{nnsd7f77}). This completes the proof.
\end{proof}

\subsection{Kolmogorov distance}
We recall that the Kolmogorov distance between the law of $F$ and standard normal law $N$ is defined by
$$d_{K}(F,N)=\sup\limits_{x\in \mathbb{R}}|P(F\leq x)-P(N\leq x)|.$$
Fixed $x\in \mathbb{R},$ we consider the Stein equation
\begin{equation}\label{stein21}
f'(z)-zf(z)=\ind_{\{z\leq x\}}-P(N\leq x),\,\,\,z\in \mathbb{R}.
\end{equation}
It is known from Lemma 2.3 in \cite{Chen2011} that the equation (\ref{stein21}) admits a unique solution $f_x(z)$ given by
$$f_x(z)=e^{\frac{z^2}{2}}\int_{-\infty}^ze^{-\frac{y^2}{2}}(\ind_{\{y\leq x\}}-P(N\leq x))dy,\,\,\,z\in \mathbb{R}.$$
Moreover, this solution satisfies
$$0<f_x(z)\leq \sqrt{2\pi}/4,\,\,\,|f'_x(z)|\leq 1\,\,\,\forall\,\,z\in \mathbb{R}$$
and
\begin{equation}\label{hklqt5}
|(w+u)f_x(w+u)-(w+v)f_x(w+v)|\leq (|w|+\sqrt{2\pi}/4)(|u|+|v|)
\end{equation}
for all $w,u,v\in \mathbb{R}.$

We have the following chain rule.
\begin{lem}\label{1lods3d3}Let $f_x$ be the solution of the Stein equation (\ref{stein21}). For any random variable $F=F(X)\in L^2(P),$ we have
\begin{align*}
\mathfrak{D}_i f_x(F)&=f'_x(F) \mathfrak{D}_iF+R^{(1)}_{i,f_x}+R^{(2)}_{i,f_x},\,\,i\geq 1,
\end{align*}
where the remainder terms $R^{(1)}_{i,f_x}$ and $R^{(2)}_{i,f_x}$ are bounded by
\begin{align*}
&|R^{(1)}_{i,f_x}|\leq(|F|+\frac{\sqrt{2\pi}}{4})(\mathfrak{d}_iF)^2,\\
&|R^{(2)}_{i,f_x}|\leq\mathfrak{D}_iF\mathfrak{D}_i\ind_{\{F> x\}}+E_i[\mathfrak{D}_iF\mathfrak{D}_i\ind_{\{F> x\}}].
\end{align*}
\end{lem}
\begin{proof}We have
\begin{align*}
f_x(F)-f_x(T_iF)&=\int_{T_iF-F}^0 f'_x(F+t)dt\\
&=f'_x(F)(F-T_iF)+\int_{T_iF-F}^0 [f'_x(F+t)-f'_x(F)]dt,\,\,i\geq 1.
\end{align*}
Since $f_x$ is the solution of the Stein equation (\ref{stein21}), we have
$$f'_x(F)=Ff_x(F)+\ind_{\{F\leq x\}}-P(N\leq x)$$
and
$$f'_x(F+t)=(F+t)f_x(F+t)+\ind_{\{F+t\leq x\}}-P(N\leq x)\,\,\forall\,t \in \mathbb{R}.$$
We therefore obtain
$$f'_x(F+t)-f'_x(F)=(F+t)f_x(F+t)-Ff_x(F)+\ind_{\{F+t\leq x\}}-\ind_{\{F\leq x\}}$$
and
\begin{align*}
f_x(F)-f_x(T_iF)=f'_x(F)(F-T_iF)+&\int_{T_iF-F}^0 [(F+t)f'_x(F+t)-Ff'_x(F)]dt\\
&+\int_{T_iF-F}^0(\ind_{\{F+t\leq x\}}-\ind_{\{F\leq x\}})dt,\,\,i\geq 1.
\end{align*}
Taking the expectation with respect to $X'_i$ yields
\begin{align*}
\mathfrak{D}_if_x(F)&=f'_x(F)\mathfrak{D}_iF+E'_i\left[\int_{T_iF-F}^0 [(F+t)f'_x(F+t)-Ff'_x(F)]dt\right]\\
&\hspace{3cm}+E'_i\left[\int_{T_iF-F}^0(\ind_{\{F+t\leq x\}}-\ind_{\{F\leq x\}})dt\right]\\
&=:f'_x(F)\mathfrak{D}_iF+R^{(1)}_{i,f_x}+R^{(2)}_{i,f_x},\,\,i\geq 1.
\end{align*}
Using (\ref{hklqt5}) we arrive at
\begin{align*}
|R^{(1)}_{i,f_x}|&\leq E'_i\big|\int_{T_iF-F}^0 [(F+t)f'_x(F+t)-Ff'_x(F)]dt\big|\\
&\leq E'_i\left[\int_{(T_iF-F)\wedge 0}^{(T_iF-F)\vee 0}(|F|+\frac{\sqrt{2\pi}}{4})|t|dt\right]\\
&=(|F|+\frac{\sqrt{2\pi}}{4})\frac{E'_i(F-T_iF)^2}{2}\\
&=(|F|+\frac{\sqrt{2\pi}}{4})(\mathfrak{d}_iF)^2,\,\,i\geq 1.
\end{align*}
In order to bound $R^{(2)}_{i,f_x},$ we observe that
$$0\leq \int_{T_iF-F}^0(\ind_{\{F+t\leq x\}}-\ind_{\{F\leq x\}})dt\leq (\ind_{\{F> x\}}-T_i\ind_{\{F> x\}})(F-T_iF).$$
Indeed, if $T_iF-F\leq 0$ then
\begin{align*}
0\leq \int_{T_iF-F}^0(\ind_{\{F+t\leq x\}}-\ind_{\{F\leq x\}})dt&\leq \int_{T_iF-F}^0(\ind_{\{F+T_iF-F\leq x\}}-\ind_{\{F\leq x\}})dt\\
&=(\ind_{\{T_iF\leq x\}}-\ind_{\{F\leq x\}})(F-T_iF)\\
&=(\ind_{\{F> x\}}-\ind_{\{T_iF> x\}})(F-T_iF)
\end{align*}
and if $T_iF-F> 0$ then
\begin{align*}
0\leq \int_{T_iF-F}^0(\ind_{\{F+t\leq x\}}-\ind_{\{F\leq x\}})dt&=\int^{T_iF-F}_0(\ind_{\{F\leq x\}}-\ind_{\{F+t\leq x\}})dt\\
&=\int^{T_iF-F}_0(\ind_{\{F+t>x\}}-\ind_{\{F> x\}})dt\\
&\leq \int^{T_iF-F}_0(\ind_{\{F+T_iF-F> x\}}-\ind_{\{F> x\}})dt\\
&=(\ind_{\{F> x\}}-\ind_{\{T_iF> x\}})(F-T_iF).
\end{align*}
Thus we obtain the following estimates
\begin{align*}
0\leq R^{(2)}_{i,f_x}&\leq E'_i[(\ind_{\{F> x\}}-T_i\ind_{\{F> x\}})(F-T_iF)]\\
&=\frac{E'_i[(F+\ind_{\{F> x\}}-T_iF-T_i\ind_{\{F> x\}})^2]-E'_i[(F-\ind_{\{F> x\}}-T_iF+T_i\ind_{\{F> x\}})^2]}{4}\\
&=\frac{(\mathfrak{d}_i(F+\ind_{\{F> x\}}))^2-(\mathfrak{d}_i(F-\ind_{\{F> x\}}))^2}{2}\\
&=\mathfrak{D}_iF\mathfrak{D}_i\ind_{\{F> x\}}+E_i[\mathfrak{D}_iF\mathfrak{D}_i\ind_{\{F> x\}}]\,\,\text{by Proposition \ref{mcsk4}, $(iv)$}.
\end{align*}
The proof of Lemma is complete.
\end{proof}
The next main result of the present paper is an explicit bound on Kolmogorov distance.
\begin{thm}\label{mld2sk}Let $F=F(X)$ be in $L^2(P)$ with mean zero, we have
\begin{align}
d_K(F,N)&\leq E|1-Z|+B_1+B_2\label{ffjw2}\\
&\leq |1-E[F^2]|+\sqrt{Var\left(Z\right)}+B_1+B_2,\nonumber
\end{align}
where $Z:=\sum\limits_{i=1}^{\infty}  \mathfrak{D}_iF E[\mathfrak{D}_iF|\mathcal{F}_i],$ and
$$B_1:=\sup\limits_{x\in \mathbb{R}}E\left[\sum\limits_{i=1}^{\infty} (\mathfrak{D}_iF\mathfrak{D}_i\ind_{\{F> x\}}+E_i[\mathfrak{D}_iF\mathfrak{D}_i\ind_{\{F> x\}}])|E[\mathfrak{D}_iF|\mathcal{F}_i]|\right],$$
$$B_2:=E\left[\sum\limits_{i=1}^{\infty} (|F|+\frac{\sqrt{2\pi}}{4}) (\mathfrak{d}_iF)^2\,|E[\mathfrak{D}_iF|\mathcal{F}_i]|\right].$$
\end{thm}
\begin{proof}Let $f_x$ be the solution of the Stein equation (\ref{stein21}). Thanks to Theorem \ref{lods3} and Lemma \ref{1lods3d3} we obtain
\begin{align*}
E&[Ff_x(F)]=E\left[\sum\limits_{i=1}^{\infty} \mathfrak{D}_i f_x(F) E[\mathfrak{D}_iF|\mathcal{F}_i]\right]\\
&=E[f'_x(F) Z]+E\left[\sum\limits_{i=1}^{\infty} R^{(1)}_{i,f_x} E[\mathfrak{D}_iF|\mathcal{F}_i]\right]+E\left[\sum\limits_{i=1}^{\infty} R^{(2)}_{i,f_x} E[\mathfrak{D}_iF|\mathcal{F}_i]\right].
\end{align*}
Hence, for all $x\in \mathbb{R},$
\begin{align*}
|P(F\leq x)-P(N\leq x)|&=|E[f'_x(F)]-E[Ff_x(F)]|\\
&\leq \|f'_x\|_\infty E|1-Z|+E\left[\sum\limits_{i=1}^{\infty} (|F|+\frac{\sqrt{2\pi}}{4})(\mathfrak{d}_iF)^2\,|E[\mathfrak{D}_iF|\mathcal{F}_i]|\right]\\
&+E\left[\sum\limits_{i=1}^{\infty} (\mathfrak{D}_iF\mathfrak{D}_i\ind_{\{F> x\}}+E_i[\mathfrak{D}_iF\mathfrak{D}_i\ind_{\{F> x\}}])|E[\mathfrak{D}_iF|\mathcal{F}_i]|\right].
\end{align*}
So we can finish the proof by using the fact that $\|f'_x\|_\infty\leq 1.$
\end{proof}
Because of the appearance of $\mathfrak{D}_i\ind_{\{F> x\}}$ in its expression, the quantity $B_1$ is difficult to bound in practice. In the next proposition, we provide a more convenient bound for this quantity.
\begin{prop}We have
\begin{align}
B_1&\leq \sqrt{Var(\bar{Z})}\leq \mathcal{L}_2(\bar{Z}),\label{jjkkd83}
\end{align}
where $\bar{Z}:=\sum\limits_{i=1}^{\infty} \mathfrak{D}_i(\mathfrak{D}_iF(|E[\mathfrak{D}_iF|\mathcal{F}_i]|+E_i|E[\mathfrak{D}_iF|\mathcal{F}_i]|)).$
\end{prop}
\begin{proof}Because $E_i[\mathfrak{D}_iF\mathfrak{D}_i\ind_{\{F> x\}}]=\mathfrak{D}_iF\mathfrak{D}_i\ind_{\{F> x\}}-\mathfrak{D}_i(\mathfrak{D}_iF\mathfrak{D}_i\ind_{\{F> x\}})$ we obtain
$$B_1=\sup\limits_{x\in \mathbb{R}}E\left[\sum\limits_{i=1}^{\infty} (2\mathfrak{D}_iF\mathfrak{D}_i\ind_{\{F> x\}}-\mathfrak{D}_i(\mathfrak{D}_iF\mathfrak{D}_i\ind_{\{F> x\}}))|E[\mathfrak{D}_iF|\mathcal{F}_i]|\right].$$
By using Proposition \ref{mcsk4}, $(iii)$ we deduce
\begin{align*}
B_1&=\sup\limits_{x\in \mathbb{R}}E\left[\sum\limits_{i=1}^{\infty} 2\mathfrak{D}_i(\mathfrak{D}_iF|E[\mathfrak{D}_iF|\mathcal{F}_i]|)\ind_{\{F> x\}}-\sum\limits_{i=1}^{\infty}\mathfrak{D}_iF\mathfrak{D}_i\ind_{\{F> x\}}\mathfrak{D}_i(|E[\mathfrak{D}_iF|\mathcal{F}_i]|)\right]\\
&=\sup\limits_{x\in \mathbb{R}}E\left[\sum\limits_{i=1}^{\infty} 2\mathfrak{D}_i(\mathfrak{D}_iF|E[\mathfrak{D}_iF|\mathcal{F}_i]|)\ind_{\{F> x\}}-\sum\limits_{i=1}^{\infty}\mathfrak{D}_i\left(\mathfrak{D}_iF\mathfrak{D}_i(|E[\mathfrak{D}_iF|\mathcal{F}_i]|)\right)\ind_{\{F> x\}}\right]\\
&=\sup\limits_{x\in \mathbb{R}}E\left[\ind_{\{F> x\}}\sum\limits_{i=1}^{\infty} \mathfrak{D}_i(\mathfrak{D}_iF(|E[\mathfrak{D}_iF|\mathcal{F}_i]|+E_i|E[\mathfrak{D}_iF|\mathcal{F}_i]|))\right]\\
&\leq E\big|\sum\limits_{i=1}^{\infty} \mathfrak{D}_i(\mathfrak{D}_iF(|E[\mathfrak{D}_iF|\mathcal{F}_i]|+E_i|E[\mathfrak{D}_iF|\mathcal{F}_i]|))\big|.
\end{align*}
We now note that each addend under the expectation is a centered random variable. Hence, (\ref{jjkkd83}) follows.
\end{proof}
\begin{rem}
Because $E_i[\mathfrak{D}_iF]=0,$ we obtain $\mathfrak{D}_i(\mathfrak{D}_iFE_i|E[\mathfrak{D}_iF|\mathcal{F}_i]|)=\mathfrak{D}_iFE_i|E[\mathfrak{D}_iF|\mathcal{F}_i]|.$
Hence, $\bar{Z}$ becomes
$$\bar{Z}=\sum\limits_{i=1}^{\infty} \mathfrak{D}_i(\mathfrak{D}_iF|E[\mathfrak{D}_iF|\mathcal{F}_i]|)+\sum\limits_{i=1}^{\infty} \mathfrak{D}_iFE_i|E[\mathfrak{D}_iF|\mathcal{F}_i]|$$
and we can obtain a further estimate for $B_1$ as follows
$$B_1\leq \sqrt{Var(\bar{Z})}\leq \sqrt{Var(\bar{Z}^{\ast})}+ \sqrt{Var(\bar{Z}^{\ast\ast})},$$
where $\bar{Z}^{\ast}:=\sum\limits_{i=1}^{\infty} \mathfrak{D}_i(\mathfrak{D}_iF|E[\mathfrak{D}_iF|\mathcal{F}_i]|)$ and $\bar{Z}^{\ast\ast}:=\sum\limits_{i=1}^{\infty} \mathfrak{D}_iFE_i|E[\mathfrak{D}_iF|\mathcal{F}_i]|.$

\end{rem}
\begin{prop} We have
\begin{equation}\label{aa1}
B_2\leq \frac{7}{2}(\sum\limits_{i=1}^\infty \sqrt{E|\mathfrak{D}_iF|^4})^{\frac{1}{2}}\sum\limits_{i=1}^\infty(E|\mathfrak{D}_iF|^4)^{\frac{3}{4}}+\frac{\sqrt{2\pi}}{4}\mathcal{L}_3(F),
\end{equation}
In particular, when $F=F(X_1,...,X_n)$ is a function of only $n$ independent random variables, we have
\begin{equation}\label{aa3}
B_2\leq 3\sqrt{n}\mathcal{L}_4(F)+\frac{\sqrt{2\pi}}{4}\mathcal{L}_3(F).
\end{equation}
\end{prop}
\begin{proof}The estimates (\ref{ode3wy}) give us the bound
$$B_2\leq \sum\limits_{i=1}^{\infty}E\big[|F|(\mathfrak{d}_iF)^2\,|E[\mathfrak{D}_iF|\mathcal{F}_i]|\big]+\frac{\sqrt{2\pi}}{4}\mathcal{L}_3(F).$$
Put $B_{2,i}:=E\big[|F|(\mathfrak{d}_iF)^2\,|E[\mathfrak{D}_iF|\mathcal{F}_i]|\big].$ Then, we can rewrite $B_2$ as
$$B_2\leq \sum\limits_{i=1}^{\infty}B_{2,i}+\frac{\sqrt{2\pi}}{4}\mathcal{L}_3(F).$$
By the H\"older inequality, each addend $B_{2,i}$ can be estimated as follows
\begin{align*}
B_{2,i}&\leq (E\big[(\mathfrak{d}_iF)^4\big])^{\frac{1}{2}}(E\big[(E[\mathfrak{D}_iF|\mathcal{F}_i])^4\big])^{\frac{1}{4}} (E|F|^4)^{\frac{1}{4}}\\
&\leq (E\big[(\mathfrak{d}_iF)^4\big])^{\frac{1}{2}}(E|\mathfrak{D}_iF|^4)^{\frac{1}{4}} (E|F|^4)^{\frac{1}{4}}.
\end{align*}
By using the Cauchy-Schwarz inequality, $E\big[(\mathfrak{d}_iF)^4\big]\leq E|\mathfrak{D}_iF|^4$ because $(\mathfrak{d}_iF)^2=\frac{1}{2}[(\mathfrak{D}_iF)^2+E_i(\mathfrak{D}_iF)^2].$ Hence, we can get
$$B_{2,i}\leq (E|F|^4)^{\frac{1}{4}}(E|\mathfrak{D}_iF|^4)^{\frac{3}{4}},\,\,i\geq 1$$
and so
\begin{equation}\label{jjdl2}
B_2\leq (E|F|^4)^{\frac{1}{4}}\sum\limits_{i=1}^{\infty}(E|\mathfrak{D}_iF|^4)^{\frac{3}{4}}+\frac{\sqrt{2\pi}}{4}\mathcal{L}_3(F).
\end{equation}
It only remains to estimate $E|F|^4.$ For each $i\geq 1,$ it follows from Proposition \ref{mcsk4}, $(v)$ that
 $$\mathfrak{D}_i(F^2)=2F\mathfrak{D}_iF-(\mathfrak{D}_iF)^2-E_i[(\mathfrak{D}_iF)^2]=2F\mathfrak{D}_iF-2(\mathfrak{d}_iF)^2.$$
By using the Cauchy-Schwarz inequality again, we obtain
\begin{align*}
E|\mathfrak{D}_i(F^2)|^2&\leq 8E|F\mathfrak{D}_iF|^2+8E[(\mathfrak{D}_iF)^4]\\
&\leq 8\sqrt{E|F|^4} \sqrt{E|\mathfrak{D}_iF|^4}+8E|\mathfrak{D}_iF|^4,\,\,i\geq 1.
\end{align*}
By the Efron-Stein inequality (\ref{sopq1}) we have
\begin{align*}
E|F|^4&=Var(F^2)+Var(F)^2\leq \sum\limits_{i=1}^\infty E|\mathfrak{D}_i(F^2)|^2+Var(F)^2\\
&\leq 8\sqrt{E|F|^4}\sum\limits_{i=1}^\infty \sqrt{E|\mathfrak{D}_iF|^4}+8\sum\limits_{i=1}^\infty E|\mathfrak{D}_iF|^4+Var(F)^2,
\end{align*}
which leads us to the following
$$
\sqrt{E|F|^4}\leq 4\sum\limits_{i=1}^\infty \sqrt{E|\mathfrak{D}_iF|^4}+\sqrt{16(\sum\limits_{i=1}^\infty \sqrt{E|\mathfrak{D}_iF|^4})^2+8\sum\limits_{i=1}^\infty E|\mathfrak{D}_iF|^4+Var(F)^2}.
$$
This, together with the elementary inequality $\sqrt{a+b}\leq \sqrt{a}+\sqrt{b},$ yields
\begin{align}
\sqrt{E|F|^4}&\leq (8+2\sqrt{2})\sum\limits_{i=1}^\infty \sqrt{E|\mathfrak{D}_iF|^4}+Var(F)\nonumber\\
&\leq (9+2\sqrt{2})\sum\limits_{i=1}^\infty \sqrt{E|\mathfrak{D}_iF|^4}.\label{hh3s}
\end{align}
In the last inequality we used the fact that $Var(F)\leq \sum\limits_{i=1}^\infty E|\mathfrak{D}_iF|^2\leq \sum\limits_{i=1}^\infty \sqrt{E|\mathfrak{D}_iF|^4}.$ Combining (\ref{jjdl2}) and (\ref{hh3s}) we obtain (\ref{aa1}).


We now consider the case, where $F$ is a function of only $n$ independent random variables. Since $\mathfrak{D}_iF=0$ for all $i\geq n+1,$ the estimate (\ref{jjdl2}) reduces to
$$B_2\leq (E|F|^4)^{\frac{1}{4}}\sum\limits_{i=1}^{n}(E|\mathfrak{D}_iF|^4)^{\frac{3}{4}}+\frac{\sqrt{2\pi}}{4}\mathcal{L}_3(F).$$
We observe that $(E[F|\mathcal{F}_i])_{1\leq i\leq n}$ is a martingale with $F=E[F|\mathcal{F}_n].$ Hence, the Burkholder's inequality \cite{Burkholder1988} implies
\begin{align*}
E|F|^4&\leq 3^4 E\left(\sum\limits_{i=1}^{n}(E[F|\mathcal{F}_i]-E[F|\mathcal{F}_{i-1}])^2\right)^{2}\\
&= 3^4E\left(\sum\limits_{i=1}^{n}(E[\mathfrak{D}_iF|\mathcal{F}_i])^2\right)^{2}.
\end{align*}
So we can obtain the following
\begin{align*}
E|F|^4
&\leq3^4 n \sum\limits_{i=1}^{n}E|E[\mathfrak{D}_iF|\mathcal{F}_i]|^4\\
&\leq  3^4 n \sum\limits_{i=1}^{n}E|\mathfrak{D}_iF|^4\\
&=3^4 n \mathcal{L}_4(F).
\end{align*}
On the other hand, we use the elementary inequality $|a_1|+...+|a_n|\leq n^{1-\frac{1}{p}}(|a_1|^p+...+|a_n|^p)^{\frac{1}{p}}$ with $p=\frac{4}{3}$ to get
$$\sum\limits_{i=1}^{n}(E|\mathfrak{D}_iF|^4)^{\frac{3}{4}}\leq n^{\frac{1}{4}}(\mathcal{L}_4(F))^{\frac{3}{4}}.$$
Consequently,
\begin{align*}B_2&\leq 3n^{\frac{1}{4}}(\mathcal{L}_4(F))^{\frac{1}{4}}n^{\frac{1}{4}}(\mathcal{L}_4(F))^{\frac{3}{4}}+\frac{\sqrt{2\pi}}{4}\mathcal{L}_3(F)\\
&=3\sqrt{n}\mathcal{L}_4(F)+\frac{\sqrt{2\pi}}{4}\mathcal{L}_3(F).
\end{align*}
The proof is complete.
\end{proof}
From the above computations, we obtain the following corollary.
\begin{cor}\label{oold1}Let $F=F(X)$ be in $L^2(P)$ with mean zero and variance $\sigma^2>0.$ Then
\begin{align}
d_K&(\sigma^{-1}F,N)\nonumber\\
&\leq  \frac{\sqrt{Var(Z)}}{\sigma^2}+ \frac{\sqrt{Var(\bar{Z})}}{\sigma^2}+ \frac{7(\sum\limits_{i=1}^\infty \sqrt{E|\mathfrak{D}_iF|^4})^{\frac{1}{2}}\sum\limits_{i=1}^\infty(E|\mathfrak{D}_iF|^4)^{\frac{3}{4}}}{2\sigma^4}+\frac{\sqrt{2\pi}}{4\sigma^3}
\mathcal{L}_3(F)\label{jjddm2}\\
&\leq  \frac{\sqrt{\mathcal{L}_2(Z)}}{\sigma^2}+ \frac{\sqrt{\mathcal{L}_2(\bar{Z})}}{\sigma^2}+ \frac{7(\sum\limits_{i=1}^\infty \sqrt{E|\mathfrak{D}_iF|^4})^{\frac{1}{2}}\sum\limits_{i=1}^\infty(E|\mathfrak{D}_iF|^4)^{\frac{3}{4}}}{2\sigma^4}+\frac{\sqrt{2\pi}}{4\sigma^3}
\mathcal{L}_3(F),
\end{align}
where $Z:=\sum\limits_{i=1}^{\infty}  \mathfrak{D}_iF E[\mathfrak{D}_iF|\mathcal{F}_i]$ and $\bar{Z}:=\sum\limits_{i=1}^{\infty} \mathfrak{D}_i(\mathfrak{D}_iF(|E[\mathfrak{D}_iF|\mathcal{F}_i]|+E_i|E[\mathfrak{D}_iF|\mathcal{F}_i]|)).$
\end{cor}
When $F$ is a function of a finite number of independent random variables,  we can also use (\ref{aa3}) for bounding $B_2.$
\begin{cor}\label{oold1b}Let $F=F(X_1,...,X_n)$ be a function of $n$ independent random variables. Assume that $E[F]=0$ and $E[F^2]=\sigma^2\in (0,\infty).$ Then
\begin{align*}
d_K(\sigma^{-1}F,N)&\leq  \frac{\sqrt{Var(Z)}}{\sigma^2}+ \frac{\sqrt{Var(\bar{Z})}}{\sigma^2}+ \frac{3\sqrt{n}}{\sigma^4}\mathcal{L}_4(F)+\frac{\sqrt{2\pi}}{4\sigma^3}\mathcal{L}_3(F)\\
&\leq  \frac{\sqrt{\mathcal{L}_2(Z)}}{\sigma^2}+ \frac{\sqrt{\mathcal{L}_2(\bar{Z})}}{\sigma^2}+ \frac{3\sqrt{n}}{\sigma^4}\mathcal{L}_4(F)+\frac{\sqrt{2\pi}}{4\sigma^3}\mathcal{L}_3(F),
\end{align*}
where $Z:=\sum\limits_{i=1}^{n}  \mathfrak{D}_iF E[\mathfrak{D}_iF|\mathcal{F}_i]$ and $\bar{Z}:=\sum\limits_{i=1}^{n} \mathfrak{D}_i(\mathfrak{D}_iF(|E[\mathfrak{D}_iF|\mathcal{F}_i]|+E_i|E[\mathfrak{D}_iF|\mathcal{F}_i]|)).$
\end{cor}
Notice that we also have
$$B_2\leq \sqrt{E[F^2]} \sum\limits_{i=1}^\infty\sqrt{E|\mathfrak{D}_iF|^6}+\frac{\sqrt{2\pi}}{4}\mathcal{L}_3(F).$$
Indeed, by the H\"older inequality
\begin{align*}
B_{2,i}&\leq \sqrt{E[F^2]}\sqrt{E\big[(\mathfrak{d}_iF)^4\,|E[\mathfrak{D}_iF|\mathcal{F}_i]|^2\big]}\\
&\leq \sqrt{E[F^2]}(E\big[(\mathfrak{d}_iF)^6\big])^{\frac{1}{3}}(E|E[\mathfrak{D}_iF|\mathcal{F}_i]|^6)^{\frac{1}{6}} \\
&\leq  \sqrt{E[F^2]} \sqrt{E|\mathfrak{D}_iF|^6},\,\,i\geq 1.
\end{align*}
We therefore obtain
\begin{cor}\label{o7old1q}Let $F=F(X)$ be in $L^2(P)$ with mean zero and variance $\sigma^2>0.$ Then
\begin{align}
d_K(\sigma^{-1}F,N)
&\leq  \frac{\sqrt{Var(Z)}}{\sigma^2}+ \frac{\sqrt{Var(\bar{Z})}}{\sigma^2}+ \frac{\sum\limits_{i=1}^\infty \sqrt{E|\mathfrak{D}_iF|^6}}{\sigma^3}+\frac{\sqrt{2\pi}}{4\sigma^3}
\mathcal{L}_3(F)\label{jjddm2bbq}
\end{align}
where $Z$ and $\bar{Z}$ are as in Corollary \ref{oold1}.
\end{cor}
\begin{rem}Comparing with Theorem 3.1 in \cite{Krokowski2016} and Theorem 4.2 in \cite{Rey2017}, our Corollaries provide new Berry-Esseen bounds for Rademacher functionals and for the functions of a finite number of independent random variables.
\end{rem}

\begin{exam}\label{dhslqo} Let $\varepsilon:=(\varepsilon_1,\varepsilon_2,...)$ be an independent Rademacher sequence, $P(\varepsilon_i=\pm1)=\frac{1}{2},\,\,i\geq 1.$ We consider the normalized series
$$F=\sum\limits_{i=1}^\infty a_i\varepsilon_i,$$
where $a_i,\,\,i\geq 1$ are real numbers such that $\sigma^2=\sum\limits_{n=1}^\infty a^2_i=1.$ Theorem 4.1 in \cite{Krokowski2016} provides the following bound on Kolmogorov distance
$$d_K(F,N)\leq 2\sum\limits_{i=1}^\infty |a_i|^3+\sup\limits_{x\in \mathbb{R}}\sum\limits_{i=1}^\infty |a_i|^2P\left(x-|a_i|\leq\sum\limits_{k=1,k\neq i}^\infty a_k\varepsilon_k \leq x+|a_i|\right),$$
where the probabilities in the second term are unknown. By applying Corollary \ref{o7old1q} we obtain the following
$$d_K(F,N)\leq 2\sum\limits_{i=1}^\infty |a_i|^3+2(\sum\limits_{n=1}^\infty a^4_i)^{\frac{1}{2}}.$$
Indeed, by the straightforward computations, we have
$$\mathfrak{D}_iF=a_i\varepsilon_i=E[\mathfrak{D}_iF|\mathcal{F}_i],$$
$$Z=1,\,\,\,\bar{Z}=\sum\limits_{n=1}^\infty 2|a_i|a_i\varepsilon_i,\,\,\,Var(Z)=0,\,\,\,Var(\bar{Z})=\sum\limits_{n=1}^\infty 4a^4_i,$$
$$\sum\limits_{i=1}^\infty \sqrt{E|\mathfrak{D}_iF|^6}=\sum\limits_{n=1}^\infty |a_i|^3,\,\,\,\mathcal{L}_3(F)=\sum\limits_{n=1}^\infty |a_i|^3.$$
\end{exam}
\begin{exam}\label{dh5slqo} Let $\varepsilon:=(\varepsilon_1,\varepsilon_2,...)$ be as in Example \ref{dhslqo} and $A=(a_{ij})_{n\times n}$ be a real symmetric
matrix ($n$ can be infinite). We consider the quadratic form
$$F=\sum\limits_{1\leq i\leq j\leq n} a_{ij}\varepsilon_i\varepsilon_j.$$
Since $E[F]=\sum\limits_{i=1}^na_{ii},$ we can and will assume that $a_{ii}=0$ for all $i.$ Under this assumption, we have $E[F]=0,\sigma^2=Var(F)=\sum\limits_{1\leq i\leq j\leq n} a^2_{ij}$ and
$$\mathfrak{D}_iF=\varepsilon_i\sum\limits_{j=1}^n a_{ij}\varepsilon_j\,\,\,\text{and}\,\,\,E[\mathfrak{D}_iF|\mathcal{F}_i]=\varepsilon_i\sum\limits_{j=1}^{i} a_{ij}\varepsilon_j,\,\,i\geq 1.$$
Hence, $Z=\sum\limits_{i=1}^n \mathfrak{D}_iFE[\mathfrak{D}_iF|\mathcal{F}_i]=\sum\limits_{j,k=1}^n \big(\sum\limits_{i=k}^na_{ij} a_{ik}\big)\varepsilon_j\varepsilon_k.$ It is easy to see that
\begin{equation}\label{quadratic3}
Var(Z)\leq 2\sum\limits_{j,k=1}^n \big(\sum\limits_{i=k}^na_{ij} a_{ik}\big)^2.
\end{equation}
In the remainder of this example, $c$ will denote a generic constant that does not depend on anything else and the value of $c$ may change from line to line.
By using the same arguments as in the proof of (\ref{hh3s}), we get
\begin{equation}\label{ppqqg3}
E|\mathfrak{D}_iF|^4=E\big|\sum\limits_{j=1}^n a_{ij}\varepsilon_j\big|^4\leq c\big(\sum\limits_{j=1}^n a^2_{ij}\big)^{2},\,\,i\geq 1
\end{equation}
and so $E|\mathfrak{D}_iF|^3\leq (E|\mathfrak{D}_iF|^4)^{\frac{3}{4}}\leq c\big(\sum\limits_{j=1}^n a^2_{ij}\big)^{\frac{3}{2}}$ and
\begin{equation}\label{quadratic4}
\mathcal{L}_3(F)\leq  c\sum\limits_{i=1}^n\big(\sum\limits_{j=1}^n a^2_{ij}\big)^{\frac{3}{2}}.
\end{equation}
Note that, when $n<\infty,$ we can use Khintchine's inequality \cite{Haagerup1981} to get the best constant $c=2^{3/2}\Gamma(2)/\sqrt{\pi}\leq 1.6.$ As a consequence, the bound (\ref{de2q1}) gives us
\begin{equation}\label{quadratic5}d_W(\sigma^{-1}F,N)\leq c\left\{\sqrt{\frac{2\sum\limits_{j,k=1}^n \big(\sum\limits_{i=k}^na_{ij} a_{ik}\big)^2}{\sigma^4}}+\frac{\sum\limits_{i=1}^n \big(\sum\limits_{j=1}^n a^2_{ij}\big)^{\frac{3}{2}}}{\sigma^3}\right\}.
\end{equation}
We now observe that $\bar Z=\sum\limits_{i=1}^{n}\bar{Z}_i,$ where
$$\bar{Z}_i:=\mathfrak{D}_i(\mathfrak{D}_iF(|E[\mathfrak{D}_iF|\mathcal{F}_i]|+E_i|E[\mathfrak{D}_iF|\mathcal{F}_i]|))=2\varepsilon_i\sum\limits_{j=1}^n a_{ij}\varepsilon_j\big|\sum\limits_{j=1}^{i} a_{ij}\varepsilon_j\big|,\,\,i\geq 1.$$
Once again, by using the same arguments as in the proof of (\ref{hh3s}) we can obtain
$$
E|\bar{Z}_i|^2\leq 4\big(E\big|\sum\limits_{j=1}^n a_{ij}\varepsilon_j\big|^4E\big|\sum\limits_{j=1}^{i} a_{ij}\varepsilon_j\big|^4\big)^{\frac{1}{2}}\leq c\big(\sum\limits_{j=1}^n a^2_{ij}\big)^2,\,\,i\geq 1.
$$
For $i\neq i',$ we consider the random variable
$$\bar{Z}^{i'}_i:=2\varepsilon_i\sum\limits_{j=1,j\neq i'}^n a_{ij}\varepsilon_j\big|\sum\limits_{j=1,j\neq i'}^{i} a_{ij}\varepsilon_j\big|.$$
We can verify that
$$E[\bar{Z}_i\bar{Z}_{i'}^i]=E[\bar{Z}_{i'}\bar{Z}^{i'}_i]=E[\bar{Z}^{i'}_i\bar{Z}_{i'}^i]=0,\,\,\,i\neq i'.$$
Hence,
$$E[\bar{Z}_i\bar{Z}_{i'}]=E[(\bar{Z}_i-\bar{Z}^{i'}_i)(\bar{Z}_{i'}-\bar{Z}_{i'}^i)]\leq \frac{1}{2}E|\bar{Z}_i-\bar{Z}^{i'}_i|^2+\frac{1}{2}E|\bar{Z}_{i'}-\bar{Z}_{i'}^i|^2,\,\,\,i\neq i'.$$
Without loss of generality we assume that $i'<i.$ Then, by the definition of $\bar{Z}_i$ and $\bar{Z}^{i'}_i,$ we have
$$|\bar{Z}_i-\bar{Z}^{i'}_i|\leq 2\big(\big|a_{ii'}\varepsilon_{i'}\sum\limits_{j=1,j\neq i'}^n a_{ij}\varepsilon_j\big|+\big|a_{ii'}\varepsilon_{i'}\sum\limits_{j=1,j\neq i'}^{i} a_{ij}\varepsilon_j\big|+a^2_{ii'}\varepsilon^2_{i'}\big),$$
$$|\bar{Z}_{i'}-\bar{Z}_{i'}^i|= 2\big|a_{i'i}\varepsilon_{i}\sum\limits_{j=1}^{i'} a_{i'j}\varepsilon_j\big|.$$
Thus, for some positive constant $c,$
$$E|\bar{Z}_i-\bar{Z}^{i'}_i|^2\leq 4\big(a^2_{ii'}\sum\limits_{j=1,j\neq i'}^n a^2_{ij}+a^2_{ii'}\sum\limits_{j=1,j\neq i'}^{i} a^2_{ij}+a^4_{ii'}\big)\leq ca^2_{ii'}\sum\limits_{j=1}^n a^2_{ij},
$$
$$E|\bar{Z}_{i'}-\bar{Z}_{i'}^i|^2= 4a^2_{i'i}\sum\limits_{j=1}^{i'} a^2_{i'j}\leq ca^2_{i'i}\sum\limits_{j=1}^{n} a^2_{i'j}.$$
Since $Var(\bar Z) =\sum\limits_{i=1}^{n} E|\bar{Z}_i|^2+\sum\limits_{i,i'=1,i\neq i'}^{n} E[\bar{Z}_i\bar{Z}_{i'}],$ combining the above inequalities yields
\begin{equation}\label{quadratic1}
Var(\bar Z)\leq c\sum\limits_{i=1}^n\big(\sum\limits_{j=1}^n a^2_{ij}\big)^2.
\end{equation}
It follows from (\ref{ppqqg3}) that
\begin{align}
(\sum\limits_{i=1}^\infty \sqrt{E|\mathfrak{D}_iF|^4})^{\frac{1}{2}}\sum\limits_{i=1}^\infty(E|\mathfrak{D}_iF|^4)^{\frac{3}{4}}&
\leq c\big(\sum\limits_{i=1}^n\sum\limits_{j=1}^n a^2_{ij}\big)^{\frac{1}{2}}\sum\limits_{i=1}^n\big(\sum\limits_{j=1}^n a^2_{ij}\big)^{\frac{3}{2}}\nonumber\\
&=\sqrt{2}c\sigma\sum\limits_{i=1}^n\big(\sum\limits_{j=1}^n a^2_{ij}\big)^{\frac{3}{2}}.\label{quadratic2}
\end{align}
Using the bound (\ref{jjddm2}), we obtain from  (\ref{quadratic3}), (\ref{quadratic4}), (\ref{quadratic1}) and (\ref{quadratic2}) that
\begin{equation}\label{quadratic5b}
d_K(\sigma^{-1}F,N)\leq c\left\{\sqrt{\frac{2\sum\limits_{j,k=1}^n \big(\sum\limits_{i=k}^na_{ij} a_{ik}\big)^2}{\sigma^4}}+\sqrt{\frac{\sum\limits_{i=1}^n\big(\sum\limits_{j=1}^n a^2_{ij}\big)^2}{\sigma^4}}+\frac{\sum\limits_{i=1}^n \big(\sum\limits_{j=1}^n a^2_{ij}\big)^{\frac{3}{2}}}{\sigma^3}\right\}.
\end{equation}
\end{exam}
\begin{rem}We only consider the quadratic form of Rademacher random variables for illustration purpose. The bounds obtained in Example \ref{dh5slqo} can be generalized easily to the quadratic form
$$F=\sum\limits_{1\leq i\leq j\leq n} a_{ij}X_iX_j,$$
where $X_i's$ are independent ones with mean zero, $\sup\limits_{i\geq 1}E|X_i|^4<\infty$ and $A=(a_{ij})_{n\times n}$ is a real symmetric matrix with vanishing diagonal. 
\end{rem}
\begin{rem}The reader can consult Section 3.1 of \cite{Chatterjee2008} for a short review on normal approximation results for the quadratic forms. It can be seen that our computations presented in Example \ref{dh5slqo} are really simple.
\end{rem}

\subsection{Applications to locally dependent random variables}
We refer the reader to Section 4.7 and Chapter 9 of \cite{Chen2011} for the bounds in the normal approximation for finite sums of locally dependent random variables. Because we only work on the functions of independent random variables, our framework is more restrictive than that considered in \cite{Chen2011}. However, we would like to emphasize that our results are able to apply to infinite sums.

Our first application is devoted to the infinite weighted runs of arbitrary independent random variables. We notice that the normal approximation for the infinite runs was discussed first by Nourdin et al. in \cite{Nourdin2010}, Berry-Esseen bounds were obtained by Krokowski et al. in \cite{Krokowski2016}. However, only the case of Rademacher random variables was considered in these papers.


Let $X_1,X_2,...$ be a sequence of independent $\mathbb{R}$-valued random variables with means $\mu_i=E[X_i]$ and finite variances $\sigma_i^2=E[(X_i-\mu_i)^2].$ Assume that $x_0:=\sup\limits_{i\geq 1}E|X_i|^4<\infty.$ We consider the 2-run $F$ defined by
$$F:=\sum\limits_{i=1}^\infty a_{i,i+1}X_iX_{i+1},$$
where $ a_{i,i+1},i\geq 1$ are real numbers such that $E[F]$ and $Var(F)$ are finite.

Obviously, we have $E[F]=\sum\limits_{i=1}^\infty a_{i,i+1}\mu_i\mu_{i+1}.$ By the straightforward computations we obtain
$$\mathfrak{D}_1F=a_{1,2}(X_1-\mu_1)X_2,\,\,\,\,E[\mathfrak{D}_1F|\mathcal{F}_1]=a_{1,2}(X_1-\mu_1)\mu_2$$
and for $i\geq 2,$
$$\mathfrak{D}_iF=a_{i-1,i}X_{i-1}(X_i-\mu_i)+a_{i,i+1}(X_i-\mu_i)X_{i+1},$$
$$E[\mathfrak{D}_iF|\mathcal{F}_i]=a_{i-1,i}X_{i-1}(X_i-\mu_i)+a_{i,i+1}(X_i-\mu_i)\mu_{i+1}.$$
From now, we use the convention $a_{i,i+1}=0$ if $i\leq 0.$ Then, for all $i\geq 1,$ we have
$$E(E[\mathfrak{D}_iF|\mathcal{F}_i])^2=[a^2_{i-1,i}\sigma^2_{i-1}+(a_{i-1,i}\mu_{i-1}+a_{i,i+1}\mu_{i+1})^2]\sigma^2_i.$$
Recalling Corollary \ref{kkjj8}, we get
\begin{equation}\label{var11}
Var(F)=\sum\limits_{i=1}^\infty [a^2_{i-1,i}\sigma^2_{i-1}+(a_{i-1,i}\mu_{i-1}+a_{i,i+1}\mu_{i+1})^2]\sigma^2_i.
\end{equation}
\begin{prop}Consider the normalized random variable $G:=\frac{F-E[F]}{\sqrt{Var(F)}}.$ Then, it holds that
\begin{equation}\label{gfgfe}
d_W(G,N)\leq c(x_0)\left(\frac{(\sum\limits_{i=1}^\infty |a_{i,i+1}|^4)^{\frac{1}{2}}}{Var(F)}+\frac{\sum\limits_{i=1}^\infty |a_{i,i+1}|^3}{Var(F)^{3/2}}\right),
\end{equation}
\begin{equation}\label{kkddfnmqo9}
d_K(G,N)\leq c(x_0)\left(\frac{(\sum\limits_{i=1}^\infty |a_{i,i+1}|^4)^{\frac{1}{2}}}{Var(F)}+ \frac{(\sum\limits_{i=1}^\infty |a_{i,i+1}|^2)^{\frac{1}{2}}\sum\limits_{i=1}^\infty |a_{i,i+1}|^3}{Var(F)^2}+\frac{\sum\limits_{i=1}^\infty |a_{i,i+1}|^3}{Var(F)^{3/2}}\right),
\end{equation}
where $c(x_0)$ is a positive constant depending only on $x_0.$
\end{prop}
\begin{proof}
\noindent{\it Part 1. Wasserstein distance.} Since $\mathfrak{D}_iG=\frac{\mathfrak{D}_iF}{\sqrt{Var(F)}},$ Theorem \ref{ko67d3} implies that
\begin{align}\label{ls01}
d_W(G,N)\leq \sqrt{\frac{2}{\pi}}\frac{\sqrt{Var(Z)}}{Var(F)}+\frac{2\mathcal{L}_3(F)}{Var(F)^{3/2}}.
\end{align}
For all $i\geq 1,$ we have
$$E|\mathfrak{D}_iF|^3\leq 4|a_{i-1,i}|^3E|X_{i-1}(X_i-\mu_i)|^3+4|a_{i,i+1}|^3E|(X_i-\mu_i)X_{i+1}|^3.$$
By Lyapunov's inequality
$$E|\mathfrak{D}_iF|^3\leq 4|a_{i-1,i}|^38x_0^{3/2}+4|a_{i,i+1}|^38x_0^{3/2}=32x_0^{3/2}(|a_{i-1,i}|^3+|a_{i,i+1}|^3),$$
and hence,
\begin{equation}\label{ls02}
\mathcal{L}_3(F)\leq 64x_0^{3/2}\sum\limits_{i=1}^\infty |a_{i,i+1}|^3.
\end{equation}
Similarly, we have
\begin{align}
E|\mathfrak{D}_iF|^4&\leq 8|a_{i-1,i}|^4E|X_{i-1}(X_i-\mu_i)|^4+8|a_{i,i+1}|^4E|(X_i-\mu_i)X_{i+1}|^4\nonumber\\
&\leq 8|a_{i-1,i}|^4 16x_0^2+8|a_{i,i+1}|^416x_0^2\nonumber\\
&=128x_0^2(|a_{i-1,i}|^4+|a_{i,i+1}|^4).\label{kdl2w}
\end{align}
Consequently,
\begin{equation}
\mathcal{L}_4(F)\leq256 x_0^2\sum\limits_{i=1}^\infty |a_{i,i+1}|^4.
\end{equation}
To bound $Var(Z),$ we write $Z=\sum\limits_{i=1}^{\infty}  Z_{i},$ where $Z_{i}:= \mathfrak{D}_iF E[\mathfrak{D}_iF|\mathcal{F}_i].$ Notice that
\begin{align*}
E|Z_{i}|^2&\leq E|\mathfrak{D}_{i}F|^4,\,\,i\geq 1.
\end{align*}
Because the random variable $X_i$ appears only in the terms $Z_{i-1},Z_{i}$ and $Z_{i+1},$ we have
$$\mathfrak{D}_iZ=\mathfrak{D}_iZ_{i-1}+\mathfrak{D}_iZ_{i}+\mathfrak{D}_iZ_{i+1},\,\,i\geq 1.$$
Hence, we can obtain the following
\begin{align*}
E|\mathfrak{D}_iZ|^2&\leq 3(E|\mathfrak{D}_iZ_{i-1}|^2+E|\mathfrak{D}_iZ_{i}|^2+E|\mathfrak{D}_iZ_{i+1}|^2)\\
&\leq 12(E|Z_{i-1}|^2+E|Z_{i}|^2+E|Z_{i+1}|^2)\\
&\leq 12( E|\mathfrak{D}_{i-1}F|^4+ E|\mathfrak{D}_{i}F|^4+ E|\mathfrak{D}_{i+1}F|^4),\,\,i\geq 1
\end{align*}
and the Efron-Stein inequality (\ref{sopq1})
\begin{equation}\label{ls03}
Var(Z)\leq \mathcal{L}_2(Z)\leq 36\sum\limits_{i=1}^\infty E|\mathfrak{D}_{i}F|^4=36\mathcal{L}_4(F)\leq 36\times 256 x_0^2\sum\limits_{i=1}^\infty |a_{i,i+1}|^4.
\end{equation}
Combining (\ref{ls01}), (\ref{ls02}) and (\ref{ls03}) yields
$$d_W(G,N)\leq 96 \sqrt{\frac{2}{\pi}}\,x_0\frac{\sqrt{\sum\limits_{i=1}^\infty |a_{i,i+1}|^4}}{Var(F)}+128x_0^{3/2}\frac{\sum\limits_{i=1}^\infty |a_{i,i+1}|^3}{Var(F)^{3/2}}.$$
Thus (\ref{gfgfe}) is verified.

\noindent{\it Part 2. Kolmogorov distance.}  Corollary \ref{oold1} gives us the following bound
\begin{align}
d_K(G,N)&\leq  \frac{\sqrt{Var(Z)}}{Var(F)}+ \frac{\sqrt{Var(\bar{Z})}}{Var(F)}\nonumber\\
&+ \frac{7(\sum\limits_{i=1}^\infty \sqrt{E|\mathfrak{D}_iF|^4})^{\frac{1}{2}}\sum\limits_{i=1}^\infty(E|\mathfrak{D}_iF|^4)^{\frac{3}{4}}}{2Var(F)^2}+\frac{\sqrt{2\pi}}{4Var(F)^{3/2}}
\mathcal{L}_3(F).\label{gft6w}
\end{align}
To bound $Var(\bar{Z}),$ we write $\bar{Z}=\sum\limits_{i=1}^{\infty}  \bar{Z}_{i},$ where $\bar{Z}_i:=\mathfrak{D}_i(\mathfrak{D}_iF(|E[\mathfrak{D}_iF|\mathcal{F}_i]|+E_i|E[\mathfrak{D}_iF|\mathcal{F}_i]|)).$ Notice that
\begin{align*}
E|\bar{Z}_i|^2&\leq 4E|\mathfrak{D}_iF(|E[\mathfrak{D}_iF|\mathcal{F}_i]|+E_i|E[\mathfrak{D}_iF|\mathcal{F}_i]|)|^2\\
&\leq 16 E|\mathfrak{D}_iF|^4,\,\,i\geq 1.
\end{align*}
Because the random variable $X_i$ appears only in the terms $\bar{Z}_{i-1},\bar{Z}_{i}$ and $\bar{Z}_{i+1},$ we have
$$\mathfrak{D}_i\bar{Z}=\mathfrak{D}_i\bar{Z}_{i-1}+\mathfrak{D}_i\bar{Z}_{i}+\mathfrak{D}_i\bar{Z}_{i+1},\,\,i\geq 1.$$
Hence,
$$E|\mathfrak{D}_i\bar{Z}|^2\leq 12\times 16( E|\mathfrak{D}_{i-1}F|^4+ E|\mathfrak{D}_{i}F|^4+ E|\mathfrak{D}_{i+1}F|^4)$$
and
\begin{equation}\label{mksg01}
Var(\bar{Z})\leq 36\times 16\sum\limits_{i=1}^\infty E|\mathfrak{D}_{i}F|^4=36\times 16\mathcal{L}_4(F)\leq36\times 16\times 256 x_0^2\sum\limits_{i=1}^\infty |a_{i,i+1}|^4.
\end{equation}
We now estimate the third addend in the right hand side of (\ref{gft6w}). From (\ref{kdl2w}), we have
$$\sqrt{E|\mathfrak{D}_iF|^4}\leq 8\sqrt{2}x_0(|a_{i-1,i}|^2+|a_{i,i+1}|^2),\,\,\,(E|\mathfrak{D}_iF|^4)^{\frac{3}{4}}\leq32\sqrt[4]{2}x^{\frac{3}{2}}_0(|a_{i-1,i}|^3+|a_{i,i+1}|^3).$$
So it holds that
\begin{align}\frac{7(\sum\limits_{i=1}^\infty \sqrt{E|\mathfrak{D}_iF|^4})^{\frac{1}{2}}\sum\limits_{i=1}^\infty(E|\mathfrak{D}_iF|^4)^{\frac{3}{4}}}{2Var(F)^2}\leq 7\times 128\sqrt{2}x_0^2 \frac{(\sum\limits_{i=1}^\infty |a_{i,i+1}|^2)^{\frac{1}{2}}\sum\limits_{i=1}^\infty |a_{i,i+1}|^3}{Var(F)^2}.\label{mksg02}
\end{align}
Finally, we insert the estimates (\ref{ls02}), (\ref{ls03}), (\ref{mksg01}) and (\ref{mksg02}) into (\ref{gft6w}) to get (\ref{kkddfnmqo9}).
\end{proof}
As expected, we obtain the rate of convergence $\frac{1}{\sqrt{n}}$ in the next corollary.
\begin{cor} Let $X_1,X_2,...$ be a sequence of i.i.d. random variables with mean $\mu$ and variance $\sigma^2\in(0,\infty).$ Assume that $E|X_1|^4<\infty,$ then
$$G_n=\frac{X_1X_2+...+X_nX_{n+1}-n\mu^2}{\sqrt{n\sigma^4+(4n+1)\sigma^2\mu^2}}$$
converges in distribution to a standard normal random variable $N$ as $n\to\infty.$ Moreover, we have
$$d_W(G_n,N)+d_K(G_n,N)\leq \frac{c}{\sqrt{n}},$$
where $c$ is a positive constant depending only on $\sigma$ and $E|X_1|^4.$
\end{cor}
\begin{rem} We observe from (\ref{var11}) that
$$Var(F)\geq (\inf\limits_{i\geq 1}\sigma_i )^4 \sum\limits_{i=1}^\infty |a_{i,i+1}|^2.$$
Hence, (\ref{kkddfnmqo9}) recovers the bound obtained in Theorem 6.1 of \cite{Krokowski2016} for the infinite 2-run of Bernoulli sequences.
\end{rem}
\begin{rem}
Given an integer number $m\geq 2,$ we consider the infinite $m$-run $F^{(m)}$ defined by
$$F^{(m)}:=\sum\limits_{i=1}^\infty a_{i,...,i+m-1}X_i...X_{i+m-1},$$
where $ a_{i,...,i+m-1},i\geq 1$ are real numbers such that $E[F^{(m)}]$ and $Var(F^{(m)})$ are finite. Using the convention $a_{i,...,i+m-1}=0$ if $i\leq 0,$ we have
$$\mathfrak{D}_iF^{(m)}=a_{i-m+1,...,i}X_{i-m+1}...X_{i-1}(X_i-\mu_i)+...+a_{i,...,i+m-1}(X_i-\mu_i)X_{i+1}...X_{i+m-1}$$
for all $i\geq 1.$ It is easy to see that
$$E|\mathfrak{D}_iF^{(m)}|^3\leq c(m,x_0) (|a_{i-m+1,...,i}|^3+...+|a_{i,...,i+m-1}|^3),$$
$$E|\mathfrak{D}_iF^{(m)}|^4\leq c(m,x_0) (|a_{i-m+1,...,i}|^4+...+|a_{i,...,i+m-1}|^4),$$
where $c(m,x_0)$ is a positive constant depending only on $m$ and $x_0.$ Hence, for $G^{(m)}:=\frac{F^{(m)}-E[F^{(m)}]}{\sqrt{Var(F^{(m)})}},$ we can obtain the bounds that are similar to (\ref{gfgfe}) and (\ref{kkddfnmqo9}). For example, we have
$$d_W(G^{(m)},N)\leq c(m,x_0)\left(\frac{(\sum\limits_{i=1}^\infty |a_{i,...,i+m-1}|^4)^{\frac{1}{2}}}{Var(F^{(m)})}+\frac{\sum\limits_{i=1}^\infty |a_{i,...,i+m-1}|^3}{Var(F^{(m)})^{3/2}}\right).$$
\end{rem} 
Next, we use difference operators of second order to apply our bounds to general structures with local dependence. For $i,k\geq 1,$ we define
$$\mathfrak{D}_{k,i}F=\mathfrak{D}_{k}(\mathfrak{D}_{i}F)=F-E_i[F]-E_k[F]+E_k[E_i[F]].$$
Notice that $\mathfrak{D}_{i,i}F=\mathfrak{D}_{i}F$ and $\mathfrak{D}_{k,i}F=\mathfrak{D}_{i,k}F.$
\begin{prop}\label{kjgmd5} Let $F=F(X)$ be in $L^2(P)$ with mean zero and variance $\sigma^2>0.$ For each $k\geq 1,$ we define the set $A_k:=\{i:\mathfrak{D}_{k,i}F\neq 0\}$ and denote by $|A_k|$ the cardinality of $A_k.$ Then
\begin{align}
d_W(\sigma^{-1}F,N)&\leq\frac{ 2}{\sigma^2}\sqrt{\frac{2}{\pi}}\big(\sum\limits_{k=1}^{\infty}|A_k|\sum\limits_{i\in A_k}  E|\mathfrak{D}_iF|^4\big)^{\frac{1}{2}}+\frac{2}{\sigma^3}\mathcal{L}_3(F)\label{oq1nr}
\end{align}
and
\begin{align}
&d_K(\sigma^{-1}F,N)\nonumber\\
&\leq \frac{10}{\sigma^2}\big(\sum\limits_{k=1}^{\infty}|A_k|\sum\limits_{i\in A_k}  E|\mathfrak{D}_iF|^4\big)^{\frac{1}{2}}+\frac{7(\sum\limits_{i=1}^\infty \sqrt{E|\mathfrak{D}_iF|^4})^{\frac{1}{2}}\sum\limits_{i=1}^\infty(E|\mathfrak{D}_iF|^4)^{\frac{3}{4}}}{2\sigma^4}
+\frac{\sqrt{2\pi}}{4\sigma^3}\mathcal{L}_3(F).\label{oq1nr1}
\end{align}
\end{prop}
\begin{proof} The fact $\mathfrak{D}_{k,i}F=0$ implies that $\mathfrak{D}_{i}F$ does not depend on $X_k,$ so does $\mathfrak{D}_iF E[\mathfrak{D}_iF|\mathcal{F}_i].$ Hence, $ \mathfrak{D}_k(\mathfrak{D}_iF E[\mathfrak{D}_iF|\mathcal{F}_i])=0$ $\,\,\forall\,i\notin A_k$ and we obtain
$$\mathfrak{D}_kZ=\sum\limits_{i\in A_k}  \mathfrak{D}_k(\mathfrak{D}_iF E[\mathfrak{D}_iF|\mathcal{F}_i]),\,\,k\geq 1.$$
Then by the Cauchy-Schwarz inequality
\begin{align*}
E|\mathfrak{D}_kZ|^2&\leq |A_k|\sum\limits_{i\in A_k}  E|\mathfrak{D}_k(\mathfrak{D}_iF E[\mathfrak{D}_iF|\mathcal{F}_i])|^2\\
&\leq 4|A_k|\sum\limits_{i\in A_k}  E|\mathfrak{D}_iF E[\mathfrak{D}_iF|\mathcal{F}_i]|^2\\
&\leq 4|A_k|\sum\limits_{i\in A_k}  E|\mathfrak{D}_iF|^4\,\,\forall\,k\geq 1,
\end{align*}
and by the Efron-Stein inequality (\ref{sopq1})
$$Var(Z)\leq 4\sum\limits_{k=1}^{\infty}|A_k|\sum\limits_{i\in A_k}  E|\mathfrak{D}_iF|^4.$$
Consequently, (\ref{oq1nr}) follows from (\ref{de2q1}).

Similarly, we have
$$\mathfrak{D}_k \bar{Z}=\sum\limits_{i\in A_k}\mathfrak{D}_i(\mathfrak{D}_iF(|E[\mathfrak{D}_iF|\mathcal{F}_i]|+E_i|E[\mathfrak{D}_iF|\mathcal{F}_i]|))\,\,\forall\,k\geq 1,$$
\begin{align*}
E|\mathfrak{D}_k\bar{Z}|^2&\leq |A_k|\sum\limits_{i\in A_k}  E|\mathfrak{D}_k(\mathfrak{D}_i(\mathfrak{D}_iF(|E[\mathfrak{D}_iF|\mathcal{F}_i]|+E_i|E[\mathfrak{D}_iF|\mathcal{F}_i]|)))|^2\\
&\leq 4|A_k|\sum\limits_{i\in A_k}  E|\mathfrak{D}_i(\mathfrak{D}_iF(|E[\mathfrak{D}_iF|\mathcal{F}_i]|+E_i|E[\mathfrak{D}_iF|\mathcal{F}_i]|))|^2\\
&\leq 16|A_k|\sum\limits_{i\in A_k}  E|\mathfrak{D}_iF(|E[\mathfrak{D}_iF|\mathcal{F}_i]|+E_i|E[\mathfrak{D}_iF|\mathcal{F}_i]|)|^2\\
&\leq 64|A_k|\sum\limits_{i\in A_k}  E|\mathfrak{D}_iF|^4\,\,\forall\,k\geq 1
\end{align*}
and
$$Var(\bar{Z})\leq 64\sum\limits_{k=1}^{\infty}|A_k|\sum\limits_{i\in A_k}  E|\mathfrak{D}_iF|^4.$$
Hence, (\ref{oq1nr1}) follows from (\ref{jjddm2}).
\end{proof}
\begin{rem}
When $F=F(X_1,...,X_n)$ is a function of only $n$ independent random variables, the bounds (\ref{oq1nr}) and (\ref{oq1nr1}) give us the following quantitative central limits theorems
\begin{equation}\label{2k3a}
d_W(\sigma^{-1}F,N)\leq\frac{ 2\delta M^{\frac{1}{2}} n^{\frac{1}{2}}}{\sigma^2}\sqrt{\frac{2}{\pi}}+\frac{2M^{\frac{3}{4}}n}{\sigma^3},
\end{equation}
\begin{equation}\label{2k3ab}
d_K(\sigma^{-1}F,N)\leq \frac{10\delta M^{\frac{1}{2}} n^{\frac{1}{2}}}{\sigma^2}+\frac{7Mn^{\frac{3}{2}}}{2\sigma^4}
+\frac{\sqrt{2\pi}M^{\frac{3}{4}}n}{4\sigma^3},
\end{equation}
provided that $\delta:=\max\limits_{k}|A_k|$ and $M:=\max\limits_{i} E|\mathfrak{D}_iF|^4$ are of finite values.
\end{rem}
\begin{rem} It is interesting to mention here Corollary 2.4 of \cite{Rey2017}. If $X_1,...,X_n$ are i.i.d. random variables with common distribution $\mu$ and $F:\mathcal{X}^n\to \mathbb{R}$ is a symmetric mapping such that $E[F(X)^2]<\infty,$ then
$$Var(F(X))\geq n\int_{\mathcal{X}}(E[F(X)-F(x,X_2,...,X_n)])^2\mu(dx).$$
Thus, in this situation, the bounds (\ref{2k3a}) and  (\ref{2k3ab}) provide us the rate of convergence $\frac{1}{\sqrt{n}}.$
\end{rem}
Another fundamental example of structures with local dependence is the $m$-scans processes. Proposition \ref{kjgmd5} gives us the following corollary.
\begin{cor}\label{yydljq1} Let $X_1,X_2,...$ be a sequence of independent random variables with $R_i=\sum\limits_{k=0}^{m-1}X_{i+k},\,i\geq 1$ denoting their $m$-scans process. We consider the random variable
$$F=\sum\limits_{i=1}^{\infty} f_i(R_i),$$
where $f_i:\mathcal{X}\to \mathbb{R},i\geq 1$ are measurable functions such that $E[f_i(R_i)]=0$ and $E|f_i(R_i)|^4<\infty.$ Suppose that $\sigma^2=Var(F)\in(0,\infty),$ then
\begin{equation}\label{oq1nr9}
d_W(\sigma^{-1}F,N)\leq cm^3\left(\frac{\big(\sum\limits_{k=1}^{\infty}E|f_i(R_i)|^4\big)^{\frac{1}{2}}}{\sigma^2}
+\frac{\sum\limits_{k=1}^{\infty}E|f_i(R_i)|^3}{\sigma^3}\right),
\end{equation}where $c$ is an absolute constant.
\end{cor}
\begin{proof} In view of Proposition \ref{kjgmd5}, we have
$$\mathfrak{D}_iF=\sum\limits_{j=(i-m+1)\vee 1}^{i}\mathfrak{D}_i f_j(R_j)\,\,\forall\,\,i\geq 1,$$
$$\text{$A_k=\{1,...,k\}$ for $k\leq m-1$ and $A_k=\{k,...,k+m-1\}$ for $k\geq m$}.$$
Hence, (\ref{oq1nr9}) follows directly from (\ref{oq1nr}).
\end{proof}
\begin{exam}(Exceedances of the $m$-scans process) Let $X_1,X_2,...$ be a sequence of i.i.d. $\mathbb{R}$-valued random variables. For
$a \in \mathbb{R},$ the random variable $W=\sum\limits_{i=1}^{n}\ind_{\{R_i> a\}}$ counts the number of exceedances of $a$ by $\{R_i:1\leq i\leq n\}.$ The study of the asymptotics of $W$ when $n\to \infty$ is motivated by its applicability to the evaluation of the significance of observed inhomogeneities
in the distribution of markers along the length of long DNA sequences. In \cite{Dembo1996}, Dembo and Rinott obtained a Berry-Esseen bound of the best possible order. Here we are able to obtain the same rate of convergence for Wasserstein distance.

Put $p:=P(R_1>a).$ From \cite{Dembo1996} we have $E[W]=np$ and $\sigma^2:=Var(W)\geq np(1-p).$ Applying Corollary \ref{yydljq1}  to $F:=W-np$ yields
$$d_W(\sigma^{-1}F,N)\leq \frac{cm^3}{\sqrt{np(1-p)}}.$$
Notice that $f_i(r)=\ind_{\{r> a\}}-p$ for all $i=1,...,n.$ Hence, $E|f_i(R_i)|^4=(1-p)^4p+p^4(1-p)\leq p(1-p)$ and $E|f_i(R_i)|^3=(1-p)^3p+p^3(1-p)\leq p(1-p).$
\end{exam}

\section{Chen-Stein's method for Poisson approximation}
Let $\text{Pn}(\theta)$ be a Poisson random variable with parameter $\theta>0$ and $\mathbb{N}=\{0,1,2,...\}.$ In this section, we investigate Poisson approximation in Wasserstein and total variation distances for $\mathbb{N}$-valued random variables $F=F(X).$ 

Given a function $f:\mathbb{N}\to \mathbb{R}$ we define the operators
$$\Delta f(a):=f(a+1)-f(a),$$
$$\Delta^2 f(a):=\Delta(\Delta f(a))$$
and write $\|f\|_\infty:=\sup\limits_{a\in \mathbb{N}}|f(a)|.$ The key of this section is the following chain rule.
\begin{lem}\label{ppo1jk}Let $f:\mathbb{N}\to \mathbb{R}$ be a measurable function. For any $\mathbb{N}$-valued random variable $F=F(X)\in L^2(P),$ we have
\begin{align*}
\mathfrak{D}_i f(F)&=\Delta f(F) \mathfrak{D}_iF+R_{i,f},\,\,i\geq 1,
\end{align*}
where the remainder term $R_{i,f}$ satisfies the bound
$$|R_{i,f}|\leq \frac{\|\Delta^2f\|_\infty}{2}[2(\mathfrak{d}_iF)^2+\mathfrak{D}_iF],\,\,i\geq 1.$$
\end{lem}
\begin{proof}We have
\begin{align*}
\mathfrak{D}_if(F) &=E'_i[f(F)-f(T_iF)]\\
&=\Delta f(F) \mathfrak{D}_iF-E'_i[f(T_iF)-f(F)-\Delta f(F)(T_iF-F)]\\
&=:\Delta f(F) \mathfrak{D}_iF+R_{i,f},\,\,i\geq 1.
\end{align*}
It is known from the proof of Theorem 3.1 in \cite{Peccati2012} that
\begin{align*}
|f(k)-f(a)-\Delta f(a)(k-a)|&\leq \frac{\|\Delta^2f\|_\infty}{2}|(k-a)(k-a-1)|\\
&=\frac{\|\Delta^2f\|_\infty}{2}(k-a)(k-a-1)\,\,\,\forall\,\,k,a\in \mathbb{N}.
\end{align*}
Applying the above inequality we obtain the following estimate for $R_{i,f}$
\begin{align*}
|R_{i,f}|&=|E'_i[f(T_iF)-f(F)-\Delta f(F)(T_iF-F)]|\\
&\leq \frac{\|\Delta^2f\|_\infty}{2}E'_i[(T_iF-F)(T_iF-F-1)]\\
&= \frac{\|\Delta^2f\|_\infty}{2}[2(\mathfrak{d}_iF)^2+\mathfrak{D}_iF],\,\,i\geq 1.
\end{align*}
So we complete the proof of Lemma.
\end{proof}
\subsection{Total variation distance}
The total variation distance between the law of $F$ and Poisson law $\text{Pn}(\theta)$ is defined by
$$d_{TV}(F,\text{Pn}(\theta))=\sup\limits_{A\subseteq \mathbb{N}}|P(F\in A)-P(\text{Pn}(\theta)\in A)|.$$
For each $A\subseteq \mathbb{N},$ we consider the Chen-Stein equation
\begin{equation}\label{lvmlw2}
\theta f(k+1)-kf(k)=\ind_{\{k\in A\}}-P(\text{Pn}(\theta)\in A),\,\,\,k\in \mathbb{N},
\end{equation}
with $f(0)=0.$ It is known that the equation (\ref{lvmlw2}) admits a unique solution, denoted by $f_A(k).$ This solution satisfies the following estimates (see, e.g. Lemma 1.1.1 and Remark 1.1.2 in \cite{Barbour1992})
\begin{equation}\label{fdmu}
\|f_A\|_\infty\leq 1\wedge \sqrt{\frac{2}{e\theta}},\,\,\,\,\|\Delta f_A\|_\infty\leq \frac{1-e^{-\theta}}{\theta},\,\,\,\,\|\Delta^2 f_A\|_\infty\leq \frac{2-2e^{-\theta}}{\theta}.
\end{equation}

\begin{thm}\label{prove01} Let $F=F(X)$ be an $\mathbb{N}$-valued random variable in $L^2(P)$ with mean $\mu$ and variance $\sigma^2,$  we have
\begin{align}
&d_{TV}(F,\text{Pn}(\theta))\nonumber\\
&\leq\left(1\wedge \sqrt{\frac{2}{e\theta}}\right) |\theta-\mu|+ \frac{1-e^{-\theta}}{\theta} \left(E|\theta- Z| + E\left[\sum\limits_{i=1}^{\infty} [2(\mathfrak{d}_iF)^2+\mathfrak{D}_iF]\,|E[\mathfrak{D}_iF|\mathcal{F}_i]|\right]\right)\label{plo1}\\
&\leq\left(1\wedge \sqrt{\frac{2}{e\theta}}\right) |\theta-\mu|+ \frac{1-e^{-\theta}}{\theta} \left(|\theta-\sigma^2|+\sqrt{Var(Z)}+2\mathcal{L}_3(F)+\mathcal{L}_2(F)\right),\label{plo2}
\end{align}
where $Z:=\sum\limits_{i=1}^{\infty}  \mathfrak{D}_iF E[\mathfrak{D}_iF|\mathcal{F}_i].$ 
\end{thm}
\begin{proof}Let $f_A$ be the solution of the equation (\ref{lvmlw2}), we have
\begin{align}
P(F\in A)&-P(\text{Pn}(\theta)\in A)=E[\theta f_A(F+1)]-E[Ff_A(F)]\nonumber\\
&=E[\theta \Delta f_A(F)]-E\left[(F-\mu)f_A(F)\right]+E[(\theta-\mu)f_A(F)].\label{uijd4}
\end{align}
Thanks to Theorem \ref{lods3} and Lemma \ref{ppo1jk} we have
\begin{align}
E\left[(F-\mu)f_A(F)\right]&=E\left[\sum\limits_{i=1}^{\infty} \mathfrak{D}_i f_A(F) E[\mathfrak{D}_iF|\mathcal{F}_i]\right]\nonumber\\
&=E[\Delta f(F)Z]+E\left[\sum\limits_{i=1}^{\infty}R_{i,f} E[\mathfrak{D}_iF|\mathcal{F}_i]\right].\label{uijd4b}
\end{align}
Inserting the relation (\ref{uijd4b}) into (\ref{uijd4}) yields
\begin{align*}
P(F\in A)&-P(\text{Pn}(\theta)\in A)\\
&=E[(\theta-Z) \Delta f_A(F)]-E\left[\sum\limits_{i=1}^{\infty}R_{i,f_A} E[\mathfrak{D}_iF|\mathcal{F}_i]\right]+E[(\theta-\mu)f_A(F)].
\end{align*}
We therefore obtain, for all $A\subseteq \mathbb{N},$
\begin{align*}
|P(F\in A)-P(\text{Pn}(\theta)\in A)|&\leq \| f_A\|_\infty |\theta-\mu|+ \|\Delta f_A\|_\infty E|\theta- Z| \\
&+\frac{\|\Delta^2f_A\|_\infty}{2}E\left[\sum\limits_{i=1}^{\infty} [2(\mathfrak{d}_iF)^2+\mathfrak{D}_iF]\,|E[\mathfrak{D}_iF|\mathcal{F}_i]|\right],
\end{align*}
which, together with the estimates (\ref{fdmu}), gives us the bound (\ref{plo1}). So we can finish the proof because (\ref{plo2}) follows directly from (\ref{plo1}).

\end{proof}
\subsection{Wasserstein distance}
Given a function $h\in Lip(1):=\{h:\mathbb{N}\to \mathbb{R}:|h(x)-h(y)|\leq |x-y|\},$ we consider the Chen-Stein equation
\begin{equation}\label{dfqmj4}
\theta f(k+1)-kf(k)=h(k)-E[h(\text{Pn}(\theta))],\,\,k\in \mathbb{N},
\end{equation}
with $f(0)=f(1).$ It is known from Theorem 1.1 in \cite{Barbour2006} that the unique solution $f_h(k)$ of the equation (\ref{dfqmj4}) satisfies
\begin{equation}\label{ppjd3}
\|f_h\|_\infty= 1,\,\,\,\,\|\Delta f_h\|_\infty\leq 1\wedge \frac{8}{3\sqrt{2e\,\theta}},\,\,\,\,\|\Delta^2 f_h\|_\infty\leq \frac{4}{3}\wedge \frac{2}{\theta}.
\end{equation}
Our last main result is the explicit bound on Wasserstein distance between the law of $F$ and Poisson law $\text{Pn}(\theta)$ defined by
$$d_W(F,\text{Pn}(\theta)):=\sup\limits_{h\in Lip(1)}|E[h(F)]-E[h(\text{Pn}(\theta))]|.$$
\begin{thm}\label{prove02} Let $F=F(X)$ be an $\mathbb{N}$-valued random variable in $L^2(P)$ with mean $\mu$ and variance $\sigma^2,$   we have
\begin{align}
&d_{W}(F,\text{Pn}(\theta))\nonumber\\
&\leq  |\theta-\mu|
+ \left(1\wedge \frac{8}{3\sqrt{2e\,\theta}}\right)E|\theta- Z|+\left(\frac{2}{3}\wedge \frac{1}{\theta}\right)E\left[\sum\limits_{i=1}^{\infty} [2(\mathfrak{d}_iF)^2+\mathfrak{D}_iF]\,|E[\mathfrak{D}_iF|\mathcal{F}_i]|\right]\\
&\leq  |\theta-\mu|+ \left(1\wedge \frac{8}{3\sqrt{2e\,\theta}}\right)(|\theta-\sigma^2|+\sqrt{Var(Z)})+\left(\frac{2}{3}\wedge \frac{1}{\theta}\right)(2\mathcal{L}_3(F)+\mathcal{L}_2(F)),
\end{align}
where $Z:=\sum\limits_{i=1}^{\infty}  \mathfrak{D}_iF E[\mathfrak{D}_iF|\mathcal{F}_i].$
\end{thm}
\begin{proof}For each $h\in Lip(1),$ we use the same arguments as in the proof of Theorem \ref{prove01} to get
\begin{align*}
E[h(F)]&-E[h(\text{Pn}(\theta))]=E[\theta f_h(F+1)]-E[Ff_h(F)]\\
&=E[(\theta-Z) \Delta f_h(F)]-E\left[\sum\limits_{i=1}^{\infty}R_{i,f_h} E[\mathfrak{D}_iF|\mathcal{F}_i]\right]+E[(\theta-\mu)f_h(F)].
\end{align*}
We therefore obtain
\begin{align*}
|E[h(F)]-E[h(\text{Pn}(\theta))]|&\leq \| f_h\|_\infty |\theta-\mu|+ \|\Delta f_h\|_\infty E|\theta- Z| \\
&+\frac{\|\Delta^2f_h\|_\infty}{2}E\left[\sum\limits_{i=1}^{\infty} [2(\mathfrak{d}_iF)^2+\mathfrak{D}_iF]\,|E[\mathfrak{D}_iF|\mathcal{F}_i]|\right],\,\,h\in Lip(1).
\end{align*}
By the estimates (\ref{ppjd3})
\begin{align*}
|E[h(F)]-E[h(\text{Pn}(\theta))]|&\leq  |\theta-\mu|+ \left(1\wedge \frac{8}{3\sqrt{2e\,\theta}}\right)E|\theta- Z| \\
&+\left(\frac{2}{3}\wedge \frac{1}{\theta}\right)E\left[\sum\limits_{i=1}^{\infty} [2(\mathfrak{d}_iF)^2+\mathfrak{D}_iF]\,|E[\mathfrak{D}_iF|\mathcal{F}_i]|\right],\,\,h\in Lip(1).
\end{align*}
So we can finish the proof of Theorem.
\end{proof}
It can be seen that the bounds in Poisson approximation are very similar to those which one encounters in normal approximation. Hence, to ensure the conciseness of the paper, we do not consider further examples.
\noindent {\bf Acknowledgments.}  This research was funded by Viet Nam National Foundation for Science and Technology Development (NAFOSTED) under grant number 101.03-2015.15.

\end{document}